\documentclass[10pt, reqno]{amsart}

\usepackage{amsmath,amsfonts,amsthm,amssymb,color,hyperref,soul}

\usepackage{mathrsfs}

\usepackage[T1]{fontenc}

\usepackage{graphicx}

\usepackage{subfigure} 

\makeatletter
     \def\section{\@startsection{section}{1}%
     \z@{.7\linespacing\@plus\linespacing}{.5\linespacing}%
     {\bfseries
     \centering
     }}
     \def\@secnumfont{\bfseries}
     \makeatother

\usepackage{graphicx}

\newcommand{\R}{\mathbb R}
\newcommand{\RR}{\mathbb R}

\newcommand{\E}{\mathbb E}

\newcommand{\1}{{\bf 1}}
\newcommand{\e}{{\varepsilon}}

\def\R{\mathbb{R}}
\def\P{\mathbb{P}}

\def\wt{\widetilde}

\def\H{\mathfrak{H}}
\def\F{\mathscr{F} }

 \def\E{\mathbb{E}}

\def\to{\rightarrow}

\def\eqd{\overset{(d)}{=}}

 \def\e{\varepsilon}

\def\w{\mathbf{w}}

\def\H{\mathfrak{H}}

\newcommand{\HH}{\mathfrak H}

\newtheorem{theorem}{Theorem}[section]
\newtheorem{lemma}[theorem]{Lemma}
\newtheorem{proposition}[theorem]{Proposition}

\theoremstyle{definition}

\theoremstyle{remark}
\newtheorem{remark}{Remark}
\numberwithin{equation}{section}
\begin{document}

\title[Averaging 2d SWE]{Averaging 2d stochastic wave equation}

 \date{\today}

 \author[R. Bola\~ nos Guerrero]{Raul Bola\~ nos Guerrero}

\author[D. Nualart]{David Nualart} \thanks {David Nualart is supported by NSF Grant DMS 1811181.}

\author[G. Zheng]{Guangqu Zheng} \thanks{Emails: $\{{\rm rbolanos,  nualart}\}$@ku.edu,   zhengguangqu@gmail.com. }

\maketitle

\vspace{-0.5cm}

\begin{center}  {\it\small Department of Mathematics, University of Kansas} \end{center}

\begin{abstract}
We consider a 2D stochastic wave  equation driven by a   Gaussian    noise, which is temporally white and  spatially colored described by the Riesz kernel. Our first main result is the functional central limit theorem for the  spatial average of the solution. And we also establish a quantitative central limit theorem for the marginal and the rate of convergence is described by the total-variation distance. A fundamental ingredient in our proofs is the pointwise $L^p$-estimate of Malliavin derivative, which is of independent interest.

 \end{abstract}



\allowdisplaybreaks

\section{Introduction}

We consider the 2D stochastic wave equation  
\begin{equation}
\label{2dSWE}
\frac{\partial^2 u}{\partial t^2} =  \Delta u  + \sigma(u)   \dot{W},
\end{equation}
on $\R_+\times \R^2$, where  $\Delta$ is Laplacian in  the space variables and $\dot{W}$ is a Gaussian centered noise  with  covariance given by   
\begin{equation}
\E[ \dot{W}(t,x)\dot{W}(s,y)    ] =  \delta_0(t-s)  \| x-y\|^{-\beta}  \label{cov}
\end{equation}
 for any given $\beta\in(0,2)$. In other words, the driving   noise $\dot{W}$    is  white  in time and   it has an homogeneous spatial covariance described by the Riesz kernel.  Here  $ \dot{W}$  is a distribution-valued field and is a notation for $\frac {\partial ^3 W}{ \partial t \partial x_1 \partial x_2}$, where the noise $W$ will be formally introduced later.

Throughout this article, we   fix the boundary conditions 
\begin{align} \label{bdcond}
   u(0,x)=1, \quad    \frac {\partial } {\partial t} u(0,x)=0 \end{align}
and assume   $\sigma:\R\to\R$ is  Lipschitz  with Lipschitz constant $L \in( 0,\infty)$ such that $\sigma(1)\neq 0$.    It is well-known (see \emph{e.g.}  \cite{Dalang})
that  equation \eqref{2dSWE} has a unique  \emph{mild solution}, which is adapted to the filtration generated by $W$, such that $\sup\big\{  \E \big[  \vert u(t,x)\vert^2\big]:   (t,x) \in [0,T]\times \R^2\big\}      < \infty$ for any finite $T$ and  
\begin{equation}\label{mild}
 u(t,x) = 1+    \int_0^t \int_{\R^2}  G_{t-s}(x-y) \sigma(u(s,y))W(ds,dy), 
\end{equation}
where   the above stochastic integral is defined in the sense of Dalang-Walsh  (see \cite{Dalang99, Walsh}) and $G_{t-s}(x-y)$ denotes the fundamental solution to the corresponding deterministic 2D wave equation, \emph{i.e.} 
\[
G_t(x) = \frac{1}{2\pi \sqrt{t^2 - \| x\|^2}} \1_{\{\| x \| <t \}}. 
\]
Because of the choice of boundary conditions \eqref{bdcond},   $\{ u(t,x): x\in\R^2\}$ is strictly stationary  for any fixed $t>0$, meaning that the finite-dimensional distributions of  $\{u(t,x+y): x\in \R^2\}$ do not depend on $y$; see \emph{e.g.} \cite[Footnote 1]{DNZ18}.  Then it is  natural to view the solution $u(t,x)$ as a functional over the homogeneous Gaussian  random field $W$. Such  Gaussian functional has been a recurrent topic in probability theory, for example, the celebrated Breuer-Major theorem (see \emph{e.g.} \cite{BM83, CNN18, NZ-OSC}) provides the Gaussian fluctuation for  the average  of a functional subordinated to a stationary    Gaussian  random field.   Therefore, one may wonder whether or not the spatial average of $u(t,x)$ admits Gaussian fluctuation, that is,   as $R\to+\infty$
\begin{equation*} 
\text{\it does} ~   \int_ {\{\| x\| \leq R\} } (  u(t,x) -1)~dx ~\text{\it converge to $\mathcal{N}(0,1)$, after proper normalization?}
\end{equation*}
Here $t>0$ is fixed,  $u(t,x)$ solves \eqref{2dSWE} and $\mathcal{N}(0,1)$ denotes the standard normal distribution.

 Recently,  the above question has been investigated for stochastic heat equations (see \cite{CKNP19-2, HNV18, HNVZ19, NZ19}) and for the  1D stochastic wave equation (see \cite{DNZ18}). Our work can be seen as an extension of the work \cite{DNZ18} to the two-dimensional case.  In Theorem \ref{MAIN} below we provide an affirmative answer   to the above question and we will provide more literature overview in Remark \ref{rem3}.
    
      \medskip
 
Let   us first fix some notation that will be used {\it throughout  this article}.

\medskip

\noindent{\bf Notation.}  (1) The expression $a  \lesssim b$ means  
$a\leq K b$ for some immaterial constant $K$ that may vary from line to line.

 (2) $\| \cdot\|$ denotes the Euclidean norm on $\R^2$ and we write $B_R= \{ x: \| x\|\leq R\}$.   We define for each $t\in\R_+:=[0,\infty)$,
\begin{equation} \label{F_R}
F_R(t) =  \int_ {B_R}  (  u(t,x) -1 )~dx.
\end{equation}

(3) We fix $\beta\in(0,2)$   throughout this article and there are two relevant constants\footnote{Note that the quantity $\kappa_\beta$ is finite, since   $J_1( \rho)$ is uniformly bounded on $\R_+$ and equivalent to $\text{constant times}\, \rho$  as $\rho\downarrow 0$; see \emph{e.g.} \cite[Lemma 2.1]{NZ19}.} $c_\beta,\kappa_\beta$ defined by 
   \begin{align}\label{CKbeta}
 c_\beta= \dfrac{\Gamma(1- \frac{\beta}{2})}{\pi 4^{\beta/2} \Gamma(\beta/2) } ,  \qquad \kappa_\beta= \int_{\R^2}d\xi \| \xi\|^{\beta-4} J_1(\| \xi \|)^2,   
 \end{align}
where 
$J_1(\cdot)$ is the Bessel function of first kind with order $1$, given by   (see, for instance, \cite[(5.10.4)]{Lebedev72})
\begin{equation} \label{J1}
J_1(x)= \frac x{\pi} \int_0 ^\pi \sin^2\theta \cos(x\cos \theta) d\theta.
\end{equation}
 Note that    $4\pi^2c_\beta\kappa_\beta =\int_{B_1^2} \| y -z\|^{-\beta}dydz $; see Remark \ref{CKREM} below.

(4) We write $\| X\|_p$   for the $L^p(\Omega)$-norm of a real random variable $X$.

\smallskip

Now we are in a position to state our main result.

   \begin{theorem}\label{MAIN}   Recall  $F_R(t)$ defined in \eqref{F_R}.  As $R\to\infty$, the process $\big\{ R^{\frac{\beta}{2} -2} F_R(t): t\in\R_+\big\}$ converges in law to a centered Gaussian process $\mathcal{G}$ in the space $C(\R_+; \R)$ of continuous functions\footnote{The  space $C(\R_+; \R)$ is equipped with   the topology of uniform convergence on compact sets.}, where 
   \[
   \E\big[ \mathcal{G}_{t_1}  \mathcal{G}_{t_2} \big]   = 4\pi^2 c_\beta \kappa_\beta \int_0^{t_1\wedge t_2}   (t_1-s)(t_2-s) \xi^2(s) ds,
   \]
    with $\xi(s) = \E[ \sigma(u(s,0) ) ]$ and $c_\beta, \kappa_\beta$ being the two constants given in \eqref{CKbeta}.  For any fixed $t>0$,  
    \begin{align}\label{QCLT}
    d_{\rm TV}\big(  F_R(t)/\sigma_R, Z   \big) \lesssim R^{-\beta/2},
    \end{align}
  where $Z\sim \mathcal{N}(0, 1)$ and $\sigma_R := \sqrt{ {\rm Var}(F_R(t))  } > 0$ for every $R>0$.    
   \end{theorem}
   
\begin{remark} (1)  The limiting process $\mathcal{G}$ has the following stochastic integral representation:  
\[
\left\{\mathcal{G}_t:  t\in\R_+\right\} \eqd \left\{2\pi \sqrt{c_\beta \kappa_\beta} \int_0^t (t-s) \xi(s) dY_s: t\in\R_+\right\},
\]
where $\{Y_t: t\in\R_+\}$  is a standard Brownian motion.

(2) We point out that  $\sigma_R>0$ is part of our main result. Indeed, it is a consequence of our standing assumption $\sigma(1)\neq 0$. In fact, we have the following equivalences:
\[
\sigma_R=0,~\forall R>0\Leftrightarrow \exists R>0,~ s.t.~ \sigma_R=0 \Leftrightarrow \sigma(1)=0 \Leftrightarrow \lim_{R\to\infty} \sigma^2_R R^{\beta-4}=0.
\]
The proof  can be done similarly as in \cite[Lemma 3.4]{DNZ18} and by using Proposition \ref{FRCOV}.

(3) The total-variation distance $d_{\rm TV}$ induces a much stronger topology than that induced by the Fortet-Mourier distance $d_{\rm FM}$, where the latter is equivalent to that of convergence in law.  For real random variables $X, Y$,
\[
d_{\rm TV}(X, Y) :=  \sup_A \big\vert \P(X\in A) - \P(Y\in A) \big\vert,\quad  d_{\rm FM}(X, Y) := \sup_{h} \big\vert  \E [ h(X) - h(Y)  ] \big\vert,
\]
where the first supremum runs over all Borel subsets of $\R$ and the second supremum runs overs  all bounded Lipschitz functions $h$ with  $\|h\|_\infty + \| h'\|_\infty \leq 1$.  Our quantitative CLT \eqref{QCLT} is obtained by the Malliavin-Stein approach that combines Stein's method of normal approximation with Malliavin's differential calculus on a Gaussian space; see the monograph \cite{NP} for a comprehensive treatment.  One can also obtain the rate of convergence in other frequently used distances, such as the 1-Wassertein distance and Kolmogorov distance, and the corresponding bounds are of the same order as in \eqref{QCLT}.
  \end{remark}

     Now let us   sketch  a few paragraphs to  briefly illustrate our methodology in proving  Theorem \ref{MAIN}. The main ingredient is  the following fundamental estimate on the $p$-norm of the Malliavin derivative  $Du(t,x)$  of the solution $u(t,x)$. It is well-known  (see \emph{e.g.} \cite{MilletMarta}) that $Du(t,x)\in L^p(\Omega; \H)$ for any $p\in[1,\infty)$, where $\H$ is the Hilbert space associated to the noise $W$, defined as  the completion of $C^\infty_c(\R_+\times\R^2)$ under the inner product
\begin{align}
\langle f, g\rangle_\H :&= \int_{\R_+\times\R^4} f(s, y) g(s, z) \| y-z\|^{-\beta} dydz ds  \label{defH}\\
&= c_\beta \int_{\R_+\times\R^2} \F f(s, \xi) \F g(s, -\xi) \| \xi\|^{\beta-2} d\xi  ds, \label{defH2}
\end{align}
where $c_\beta$ is given in \eqref{CKbeta} and $\F f(s, \xi) = \int_{\R^2} e^{-i x\cdot \xi} f(s, x)dx$.  
     
        \begin{theorem}\label{THM}   The Malliavin derivative $Du(t,x)$ is a random function denoted by $(s,y) \mapsto D_{s,y}u(t,x)$ and  for any $p\in[2,\infty)$ and any $t>0$, the following estimates hold   for almost all $(s,y) \in [0,t] \times \R^2$:        
 \begin{align}
 G_{t-s}(x-y) \| \sigma(u_{s,y})\|_p \leq   \big\| D_{s,y} u(t,x) \big\|_p  \leq   C_{\beta, p,t,L} \kappa_{p,t} G_{t-s}(x-y), \label{IMP}
 \end{align} 
 where the constants  $C_{\beta, p,t,L}$ and $ \kappa_{p,t}$ are given in \eqref{Cbeta} and \eqref{sigmap}, respectively.
    \end{theorem}
   
\begin{remark}    Theorem \ref{THM} echoes the comment after \cite[Lemma 2.1]{HNVZ19} and generalizes  \cite[Lemma 2.2]{DNZ18} to the solution of a 2D stochastic wave equation.  Although the expression in \eqref{IMP} looks the same as in \cite[Lemma 2.2]{DNZ18}, \emph{i.e.} $L^p$-norm of the Malliavin derivative is bounded by the fundamental solution to the corresponding deterministic wave equation, we would like to emphasize that the proof in the 2D setting is much more involved and requires new techniques in dealing with the singularity of $G_{t-s}(x-y)$ while in the  1D case the fundamental solution is the bounded function $\frac{1}{2}\mathbf{1}_{\{ | x-y| < t-s \}}$. Modulo sophisticated integral estimates, our proof of Theorem \ref{THM} is treated   through a harmonious combination of tools from  
Gaussian analysis (Clark-Ocone formula, Burkholder inequality) and  Hardy-Littlewood-Sobolev's lemma.
\end{remark}   
     
     Before we proceed to explaining our proof strategy, let us provide a  brief  literature overview.
     
     \begin{remark} \label{rem3}  It was the paper \cite{HNV18}  by Huang, Nualart and Viitasaari that first  studied  spatial averages of stochastic heat equation with 1 spatial dimension driven by space-time white noise.  Soon later, the same authors and Zheng investigated the same equation in higher dimension; in their paper \cite{HNVZ19}, the spatial correlation is described by the Riesz kernel as in the present work. The above two references considered the noise that is white in time, leading to the natural martingale structure. This enables one to take advantage of It\^o calculus mentioned in previous remark. However,  when the noise is colored in time, these tools are not available any more and we should restrict ourselves to the linear equation (that is, when $\sigma(u)=u$). The linear equation, also known as the \emph{parabolic Anderson  model}, admits the explicit Wiener chaos expansions, and   in the work \cite{NZ19} by Nualart and Zheng, similar central limit theorems are established at qualitative level by using the so-called chaotic central limit theorem (see \emph{e.g.} \cite[Section 6.3]{NP}).  The authors of \cite{DNZ18} first considered the same problem for the stochastic wave equations 
     where spatial dimension is one and the driving Gaussian noise is white in time and fractional in space. Unlike in the heat setting, the fundamental wave solution differs in different dimensions and as we will see shortly, the analysis in our work is quite different from that in  \cite{DNZ18}. Here we also remark that it is natural to study the same problem for wave equations when the noise is colored in time, and it may be a hard problem to get a quantitative central limit theorem in this setting.  
     
     \end{remark}

Now let us   first sketch the main steps for the proof of Theorem \ref{MAIN} and then we will present the key steps in proving \eqref{IMP}.
     
  The typical proof of the functional CLT consists in three steps: 
  \medskip
  
 \noindent{($\textbf{S1}$)} We establish the limiting covariance structure, this is the content of Section \ref{sub41}. In particular, the variance of the spatial average $F_R(t)$ is of order $R^{4-\beta}$, as $R\to\infty$. As one will see shortly, the important part of this step is the proof of the limit \eqref{enough1}: $ \text{Cov}\big[ \sigma(u(s,y)), \sigma(u(s,z))  \big]\to 0$ as $\| y -z\|\to\infty$.  This limit is straightforward when $\sigma(u)=u$ and in the general case, we will apply the Clark-Ocone formula (see Lemma \ref{CO}) to first represent $\sigma(u(s,y))$ as a stochastic integral and then apply the It\^o's isometry in order to break the nonlinearity for further estimations.
  
  \medskip
  
\noindent{($\textbf{S2}$)} From $(\textbf{S1})$, we have the covariance structure of the limiting Gaussian process $\mathcal{G}$. Then we will prove the convergence of $\big\{ R^{\frac{\beta}{2} - 2}F_R(t): t\in\R_+\big\}$ to 
$\big\{\mathcal{G}_t: t\in\R_+\big\}$ in finite-dimensional distributions. This is made possible by the  following multivariate Malliavin-Stein bound that we borrow from     \cite[Proposition 2.3]{HNV18} (see also \cite[Theorem 6.1.2]{NP}).   We denote by $D$ the  Malliavin derivative and  by $\delta$  the adjoint operator of $D$ that is characterized by the  integration-by-parts formula
(\ref{IBP}). Moreover,  $\mathbb{D}^{1,2}$ is the Sobolev space of Malliavin differentiable random variables $X\in L^2(\Omega)$ with $\E\big[\| DX\|_\H^2\big] < \infty$ and ${\rm Dom}\delta$ is the domain of $\delta$; see Section \ref{sec2} for more details.
 
\begin{proposition}\label{lem612}
Let $F=( F^{(1)}, \dots, F^{(m)})$ be a random vector such that $F^{(i)} = \delta (v^{(i)})$ for $v^{(i)} \in {\rm Dom}\, \delta$ and
$  F^{(i)} \in \mathbb{D}^{1,2}$, $i = 1,\dots, m$. Let $Z$ be an $m$-dimensional centered Gaussian  vector with covariance  matrix $(C_{i,j})_{ 1\leq i,j\leq m} $. For any  $C^2$ function $h: \R^m \rightarrow \R$ with bounded second partial derivatives, we have
\begin{align}\label{MS612}
\big| \E [ h(F)] -\E [ h(Z)]  \big| \le \frac{m}{2}  \|h ''\|_\infty \sqrt{   \sum_{i,j=1}^m   \E \Big[ \big(C_{i,j} - \langle DF^{(i)}, v^{(j)} \rangle_{\HH}\big)^2 \Big] } \,,
\end{align}
where $\|h ''\|_\infty : = \sup\big\{   \big\vert \frac{\partial^2}{\partial x_i\partial x_j} h(x) \big\vert\,:\, x\in\RR^m\, , \, i,j=1, \ldots, m \big\}$. 
\end{proposition}

 In view of \eqref{mild}, we   write $u(t,x)-1 = \delta\big( G_{t-\bullet}(x-\ast) \sigma(u(\bullet,\ast)) \big)$ so that $F_R(t)$ can be represented as 
 \begin{align}\label{sFub}
 F_R(t) = \int_{B_R} \delta\big( G_{t-\bullet}(x-\ast) \sigma(u(\bullet,\ast)) \big)dx =  \delta\big( \varphi_{t,R}(\bullet,\ast)  \sigma(u(\bullet,\ast))\big)
 \end{align}
 by  Fubini's theorem, with  
 \begin{equation} \label{varphi}
\varphi_{t,R}(r,y) = \int_{B_R} G_{t-r}(x-y) dx;
\end{equation}
see Section \ref{sub22}. Putting $V_{t,R}(s,y) =  \varphi_{t,R}(s,y)  \sigma(u(s,y)) $, and applying the fundamental estimate \eqref{IMP}, we will establish that,   for any $t_1, t_2\in(0,\infty)$,
\begin{equation} \label{eq1}
 R^{2\beta- 8} \text{Var}\big( \langle DF_R(t_1), V_{t_2,R}\rangle_\H \big) \lesssim R^{-\beta} \,\, \text{for} \,\,  R\ge t_1
 +t_2.
\end{equation}
Then,  we will show that Proposition \ref{lem612} together with the estimate (\ref{eq1}) imply
 the convergence in law of the finite-dimensional distributions. 

The  bound \eqref{eq1} for $t_1=t_2=t$ together with the following 1D Malliavin-Stein bound
(see, {\it e.g.} \cite{HNV18, Eulalia,Zhou}) will lead to the quantitative result \eqref{QCLT}.

\begin{proposition}  
Let $F=\delta (v)$ for some $\HH$-valued random variable $v\in {\rm Dom }\, \delta$. Assume  $F\in \mathbb{D}^{1,2}$ and $\E [F^2] = 1$ and  let $Z \sim \mathcal{N}(0,1)$.  Then,
\begin{equation}
d_{\rm TV}(F, Z) \leq 2 \sqrt{{\rm Var}\big[  \langle DF, v\rangle_{\HH} \big] }\,.  \label{1DNP}
\end{equation}
\end{proposition}

\medskip

\noindent{($\textbf{S3}$)} The last step is to show  tightness,  which follows from the tightness of the processes restricted to $[0,T]$ for any finite $T$. To show the tightness of $\big\{ R^{\frac{\beta}{2} - 2}   F_R(t): t\in[0,T] \big\}$, in view of  the well-known criterion of Kolmogorov-Chentsov (see \emph{e.g.} \cite[Corollary 16.9]{OK}), it is enough to show  that   for any $p\in[2,\infty)$,
\begin{equation} \label{KC}
\| F_R(t) - F_R(s) \|_p \lesssim R^{2-\frac{\beta}{2}} \vert t-s\vert^{1/2} ~{\rm for}~ s,t\in[0,T],
\end{equation}
where the implicit constant does not depend on $t,s$ or $R$.    This will proves  Theorem \ref{MAIN}.

\medskip

 Finally let us pave the plan of proving the fundamental estimate \eqref{IMP}. The story begins with the usual {\it Picard iteration}: We define $u_0(t,x)=1$ and for $n\geq 0$,
 \begin{equation} \label{Pi}
 u_{n+1}(t,x) = 1 + \int_0^t \int_{\R^2} G_{t-s}(x-y) \sigma\big( u_n(s,y) \big) W(ds, dy).   
 \end{equation}
 It is a classic result that $u_n(t,x)$ converges in $L^p(\Omega)$ to $u(t,x)$ uniformly in $x\in\R^2$ for any $p\ge 2$; see \emph{e.g.} \cite[Theorem 4.3]{Dalang}.  
Now it has become clear  that if we assume $\sigma(1)=0$, we will end up in the trivial case where $u(t,x)\equiv 1$, in view of the above iteration. 

For each $n\geq 0$, $u_{n+1}(t,x)$ is  Malliavin differentiable, as one can show by induction on $n$. Our strategy is to first obtain the uniform estimate of  $\sup\big\{ \| D_{s,y} u_n(t,x) \|_p: n\geq 0\big\}$ and then one can hope to transfer this estimate to $\| D_{s,y}u(t,x)\|_p$.   As mentioned before, $Du(t,x)$ lives in the space $\H$ that contains generalized functions. To overcome this, we will carefully apply the following  inequality of Hardy-Littlewood-Sobolev to show $Du(t,x)$ is a random variable in $L^{\frac{4}{4-\beta}}(\R_+\times\R^2)$, with $\beta\in(0,2)$ fixed throughout this paper.

     \begin{lemma}[\textbf{Hardy-Littlewood-Sobolev}]  \label{HLS} If $1 < p < p_0 < \infty$ with $p_0^{-1} = p^{-1} - \alpha n^{-1}$, then there is    some constant $C$ that only depends on $p$, $\alpha$ and $n$, such that  
  \[
   \|   I^\alpha g  \|_{L^{p_0}(\R^n)} \leq C  \| g \| _{L^p(\R^n) },
   \]
   for any locally integrable function $g:\R^2\to\R$, where with $\alpha\in(0,n)$, 
    \[
  \big( I^\alpha g \big)(x) := \int_{\R^n} \| x - y\|^{\alpha - n} g(y) dy.
   \]
    \end{lemma} 
   
 For our purpose, with $n=2$, $\alpha= 2-\beta$, $p= 2q=4/(4-\beta)$ and $p_0 = 4/\beta$, we deduce from    H\"older's inequality that 
   \begin{align}
 \langle f, g\rangle_{\mathfrak{H}_0} &:=  \int_{\R^2} f(x) g(y) \| x- y\|^{-\beta} dxdy   \label{H_0}  \\
 & \le  \| f\|_{L^{2q}(\R^2)} \| I^{2-\beta}g \| _{L^{4/\beta}(\R^2)} \notag \\
 &  \leq C_\beta \| f \| _{L^{2q}(\R^2)} \| g \| _{L^{2q}(\R^2)}, \label{Sobolev}
   \end{align}
   for any $f,g\in L^{2q}(\R^2)$; see {\it e.g.} \cite[pages 119-120]{Stein70}.
  
  \medskip
     
Once we obtain the uniform estimate of  $\sup\big\{ \| D_{s,y} u_n(t,x) \|_p: n\geq 0\big\}$ and prove $Du(t,x)\in L^{\frac{4}{4-\beta}}(\R_+\times\R^2)$, that is, $(s,y)\longmapsto D_{s,y}u(t,x)$ is indeed a random function,
we proceed to the proof of  \eqref{IMP}.
In view of the Clark-Ocone formula (see Lemma \ref{CO}), we have $\E[ D_{s,y} u_{t,x} \vert \F_s \big] = G_{t-s}(x-y) \sigma(u(s,y) )$ almost surely, where $\{\F_s: s\in\R_+\}$ is the filtration generated by the noise; see Section \ref{sub22}. Then, the lower bound  in \eqref{IMP} follows immediately from  the conditional Jensen inequality.  The upper bound follows from the uniform estimates of  $\| D_{s,y} u_n(t,x) \|_p$
by a standard argument.      

\medskip

  Before we end this introduction, let us point out another technical difficulty in this paper. After the application of Lemma \ref{HLS} during the process of estimating  $\| D_{s,y} u_n(t,x) \|_p$, we will encounter  integrals of the form
  \begin{align}\label{intform}
\int_s^t  \left(\int_{\R^2} G^{2q}_{t-r}(x-z) G_{r-s}^{2q}(z)dz\right)^{\delta}dr
 \end{align}
where $q\in (1/2,1)$ and  $\delta \in \{1, 1/q\}$. 
 In the case of stochastic heat equation, the estimation of the above integrals  is  straightforward due to the semi-group property.   However, for the wave equation the kernel $G_t$ {\it does not } satisfy the semi-group property and the estimation of the above integrals is  quite involved.
  For  the case of the 1D stochastic wave equation,  as one can see from the paper  \cite{DNZ18}, the computations take advantage of the simple  form of the fundamental solution (\emph{i.e.} $\frac12\1_{\{  |x-y| <t-s \}}$). For our 2D case, the singularity within the fundamental solution $G_{t-s}(x-y)$ puts the technicality to another level and we have to estimate the convolution $ G_{t-r}^{2q} \ast G_{r-s}^{2q}$ by exact computations. A basic technical tool  used in this problem is the following lemma.

      \begin{lemma}\label{LEM1}  For  $0\leq s<t<\infty$, with $\| z\|= \w >0$ and $q\in(1/2,1)$, we have
\begin{align}
G_t^{2q}\ast G_s^{2q}(z) &\lesssim   {\bf 1}_{\{ \w < s\}}  \big[ t^2 - (s-\w)^2\big]^{1-2q} +       \big[ t^2 - (s+\w)^2 \big]^{1-2q } {\bf 1}_{\{ t> s+\w  \}} \notag \\
 &\quad +  {\bf 1}_{\{\vert s-\w \vert < t < s+\w  \}} \big[ (\w+s)^2 -t^2\big]^{-q+\frac12} \big[t^2- (s-\w)^2\big]^{-q+\frac12},\label{qq}
 \end{align}
where the implicit constant only depends on $q$.
\end{lemma}

The rest of this article is organized as follows:  Section 2 collects some preliminary facts for our proofs, Section 3  contains the proof of  Theorem \ref{MAIN} and Section 4 is devoted to proving the fundamental estimate \eqref{IMP}.
 
 \medskip 
 \paragraph{\bf Acknowledgement:} We are grateful to two referees for their critical comments that improved our work.
    
    \section{Preliminaries}  \label{sec2}
    
  This section provides  some preliminary results that are required for further sections. It consists of two subsections:  Section \ref{sub21} contains several important facts on the function $G_{t-s}(x-y)$ and Section \ref{sub22} is devoted to a minimal set of  results from stochastic analysis, notably the tools from Malliavin calculus.
    
    \subsection{Basic facts on the fundamental solution}\label{sub21}

Let us fix  some  more notation here. 
\medskip

\noindent{\bf Notation.} For $p\in\R$, we write $(v)_+^p =v^p$ if $v>0$ and $(v)_+^p = 0$ if $v\leq 0$. Then, we can write
\[
G_t(x) = \frac{1}{2\pi}       (t^2 - \|x\|^2)_+^{-1/2}.
\]
Recall the function  $\varphi_{t,R}(r,y)$ introduced in \eqref{varphi}:  
 \[
 \varphi_{t,R}(s,y)= \int_{B_R} G_{t-r}(x-y)dx.
 \] 
In what follows, we put together several useful facts on the function $G_{t}(z)$.

\begin{lemma}\label{LFact}   {\rm (1)}
 For  any $p\in(0,1)$ and $t>0$.
  \begin{align}
   \int_{\R^2} G^{2p}_{t}(z)  dz= \frac{  (2\pi)^{1-2p} }{2-2p} t^{2-2p } .  \label{fact01}
  \end{align} 

 {\rm (2)} For $t>s$, 
  we have $ \varphi_{t,R}(s,y)\leq (t-s)\1_{\{ \|y\|\leq R+t\}}$ and ${\displaystyle \int_{\R^2}  \varphi_{t,R}(s,y) dy   =(t-s) \pi R^2}$.
  \end{lemma}
The proof of Lemma \ref{LFact}  is omitted, as it follows from simple and exact computations.
As a consequence of  Lemma \ref{LFact}-(2), we have
 \begin{align}
  \int_{\R^2} \varphi_{t,R}(s, z+\xi) \varphi_{t,R}(s, z) dz \leq
   \pi (t-s)^2 R^2. \label{fact00}
  \end{align}
The following lemma is also a consequence of Lemma \ref{LFact}.

  \begin{lemma}\label{lemfact} For $t_1,t_2\in(0,\infty)$, we put
  \[
  \Psi_R(t_1, t_2; s) : =R^{\beta-4}  \int_{\R^4} \varphi_{t_1,R}(s,y) \varphi_{t_2,R}(s,z) \| y - z\|^{-\beta}    dydz.
  \]
  Then 
  \begin{itemize}
  \item[(i)] $\Psi_R(t_1, t_2; s)$ is uniformly bounded over $s\in[0, t_2\wedge t_1]$ and $R>0$; 
  \item[(ii)] For any $s\in[0, t_2\wedge t_1]$, 
  $
   \Psi_R(t_1, t_2; s) ~\text{converges to}~  4\pi^2 c_\beta \kappa_\beta (t_1-s)(t_2-s),
  $
  as $R\to\infty$.
  \end{itemize}
Here the quantities $c_\beta$ and $\kappa_\beta$ are given   in \eqref{CKbeta}.
\end{lemma}

  \begin{proof}   By using Fourier transform as in (\ref{defH2}), we can write
 \begin{align*}  
 &\quad \Psi_R(t_1, t_2; s) = R^{\beta-4} \int_{B_R^2} dxdx'  \int_{\R^4} G_{t_1-s}(x-y) G_{t_2-s}(x'-z) \| y - z\|^{-\beta}    dydz  \notag \\
&= c_{\beta}  R^{\beta-4} \int_{B_R^2} dxdx'  \int_{\R^2}d\xi e^{- i (x-x')\cdot \xi} \left( \frac{ \sin( (t_1-s) \| \xi \|  )  }{\| \xi \|}  \frac{ \sin( (t_2-s) \| \xi \|  )  }{\| \xi \|} \right) \| \xi \|^{\beta -2}   \\
&= c_\beta  \int_{B_1^2} dxdx'  \int_{\R^2}d\xi e^{- i (x-x')\cdot \xi} ~ \frac{ \sin( (t_1-s) \| \xi \| R^{-1}  )  }{\| \xi \| R^{-1}}  \frac{ \sin( (t_2-s) \| \xi \| R^{-1}  )  }{\| \xi \| R^{-1}}  \| \xi \|^{\beta -2} ,
 \end{align*} 
where in the last equality we made the   change of variables  $\xi \to \xi R^{-1}$.

The Fourier transform of $x\in\R^2\longmapsto \1_{\{ \| x\| \leq1 \}}$ is $\xi\in\R^2\longmapsto 2\pi \|\xi\|^{-1} J_1(\| \xi\|)$
(see, for instance,  Lemma 2.1 in \cite{NZ19}), where $J_1$ is the Bessel function of first kind with order $1$ introduced in (\ref{J1}).
 Then, we can rewrite $ \Psi_R(t_1, t_2; s)$ as 
 \[
  c_\beta   \int_{\R^2} \Big[2\pi \| \xi\|^{-1} J_1(\| \xi \|) \Big]^2 \left(  \frac{ \sin( (t_1-s) \| \xi \| R^{-1}  )  }{\| \xi \| R^{-1}}  \frac{ \sin( (t_2-s) \| \xi \| R^{-1}  )  }{\| \xi \| R^{-1}} \right) \| \xi \|^{\beta -2}d\xi.
 \]
  Since $ \sin( (t-s) \| \xi \| R^{-1} ) / (\| \xi \| R^{-1}) $ is uniformly bounded over $s\in(0,t]$ and converges to $t-s$ as $R\to\infty$, then the statement (i) holds true and
 \[
  \Psi_R(t_1, t_2; s)\xrightarrow{R\to\infty}  4\pi^2 c_\beta \kappa_\beta (t_1-s)(t_2-s).
    \]
 by the dominated convergence  theorem with the dominance condition $\kappa_\beta < \infty$.    \end{proof}
 
 \begin{remark} \label{CKREM}
  By inverting the Fourier transform, we have
  \[
 ( 2\pi)^2 c_\beta \kappa_\beta = c_\beta \int_{\R^2} (2\pi)^2 J_1(\|\xi\|)^2 \| \xi\|^{-2} \| \xi\|^{\beta-2}d\xi =\int_{B_1^2} \|y-z\|^{-\beta} dydz. 
  \]
 \end{remark}

     \subsection{Basic stochastic analysis} \label{sub22}

Let  $\H$ be defined  (see \eqref{defH} and \eqref{defH2})  as the completion of $C_c^\infty(\R_+\times\R^2)$ under the inner product 
\[
\langle f, g\rangle_\H = \int_{\R_+\times\R^4} f(s,y) g(s, z) \| y-z\|^{-\beta}dydzds ~ \text{for}~ f, g\in C^\infty_c(\R_+\times\R^2).
\]
Consider  an isonormal Gaussian process associated to the Hilbert space $\H$, denoted by $W=\big\{ W(\phi): \phi\in \H \big\}$. That is, $W$ is a centered  Gaussian family of random variables such that
$
\E\big[ W(\phi) W(\psi) \big] = \langle \phi, \psi\rangle_\H$ for any   $\phi, \psi\in \H$.
As the noise is white in time, a martingale structure naturally appears. First we define $\F_t$ to be the $\sigma$-algebra generated by $\P$-null sets and $\big\{ W(\phi): \phi\in C^\infty(\R_+\times\R^2)$ has compact support contained in $[0,t]\times\R^2 \big\}$, so we have a filtration $\mathbb{F}=\{\F_t: t\in\R_+\}$. If $\big\{\Phi(s,y): (s,y)\in\R_+\times\R^2\big\}$ is  an $\mathbb{F}$-adapted  random field such that $\E\big[ \| \Phi\|_\H^2 \big] <+\infty$, then 
\[
M_t = \int_{[0,t]\times\R^2} \Phi(s,y)W(ds,dy),
\]
interpreted as the Dalang-Walsh integral (\cite{Dalang99,Walsh}), is a square-integrable $\mathbb{F}$-martingale with quadratic variation given by
\[
\langle M \rangle_t = \int_{[0,t]\times\R^4} \Phi(s,y) \Phi(s,z) \|y-z\|^{-\beta} dydzds =\big\| \Phi(\bullet,\ast)\1_{\{ \bullet \leq t \}}\big\|_\H^2.
\]
Let us record a suitable version of Burkholder-Davis-Gundy inequality (BDG for short); see  {\it e.g.} \cite[Theorem B.1]{Khoshnevisan}.

\begin{lemma}[BDG]  If $\big\{\Phi(s,y): (s,y)\in\R_+\times\R^2\big\}$ is an adapted random field with respect to $\mathbb{F}$ such that $ \| \Phi\|_\H \in L^p(\Omega)$ for some $p\geq 2$, then
\begin{equation} \label{BDG}
\left\|  \int_{[0,t]\times\R^2} \Phi(s,y)W(ds,dy) \right\|_p^2 \leq 4p \left\| \int_{[0,t]\times\R^4} \Phi(s,z)\Phi(s,y) \| y-z\|^{-\beta} dydzds \right\|_{p/2}.
\end{equation}
\end{lemma}
We refer interested readers to the book \cite{Khoshnevisan} for a nice introduction to Dalang-Walsh's theory. For our purpose, we will often apply BDG as follows.  If $\Phi$ is $\mathbb{F}$-adapted and 
$\|  G_{t-\bullet}(x-\ast) \Phi(\bullet, \ast) \|_{\H}\in L^p(\Omega)$ for some $p\ge 2$, then BDG implies 
\begin{align} \notag
&\left\|  \int_{[0,t]\times\R^2}  G_{t-s}(x-y) \Phi(s,y)W(ds,dy) \right\|_p^2\\ \label{eq2}
 &\qquad  \leq  4p\left\|  \int_{[0,t]\times\R^4}G_{t-s}(x-z)  G_{t-s}(x-y) \Phi(s,y) \Phi(s,z) \|y-z\|^{-\beta}dsdzdy \right\|_{p/2},
\end{align}
by viewing  $\int_{[0,t]\times\R^2}  G_{t-s}(x-y) \Phi(s,y)W(ds,dy)$ as the martingale 
$$\left\{ \int_{[0,r]\times\R^2}  G_{t-s}(x-y) \Phi(s,y)W(ds,dy) : r\in[0,t] \right\} ~\text{evaluated at at time $t$.}$$

\medskip

Now let us recall some basic facts on the Malliavin calculus associated with $W$. For any unexplained notation and result,   we refer to the book \cite{Nualart}.  We denote by $C_p^{\infty}(\R^n)$ the space of smooth functions with all their partial derivatives having at most polynomial growth at infinity. Let $\mathcal{S}$ be the space of simple functionals of the form 
$F = f(W(h_1), \dots, W(h_n))
$ for $f\in C_p^{\infty}(\RR^n)$ and $h_i \in \HH$, $1\leq i \leq n$. Then, the Malliavin derivative  $DF$ is the $\HH$-valued random variable given by
\begin{align*}
DF=\sum_{i=1}^n  \frac {\partial f} {\partial x_i} (W(h_1), \dots, W(h_n)) h_i\,.
\end{align*}
 The derivative operator $D$  is   closable   from $L^p(\Omega)$ into $L^p(\Omega;  \HH)$ for any $p \geq1$ and   we define $\mathbb{D}^{1,p}$ to be the completion of $\mathcal{S}$ under the norm
$
\|F\|_{1,p} = \left(\E\big[ |F|^p \big] +   \E\big[  \|D F\|^p_\HH \big]   \right)^{1/p} \,.
$

  The {\it chain rule} for $D$ asserts that if $F_1,F_2\in\mathbb{D}^{1,2}$ and $h_1,h_2:\R\to\R$ are Lipschitz, then $h_1(F_1)h_2(F_2)\in\mathbb{D}^{1,1}$ and $h_i(F_i)\in\mathbb{D}^{1,2}$ with 
\begin{equation} \label{cr}
D\big( h_1(F_1) h_2(F_2)\big) = h_2(F_2) Y_1 DF_1 + h_1(F_1) Y_2 DF_2,
\end{equation}
where  $Y_i$ is some  $\sigma\{ F_i\}$-measurable random variable  bounded by the Lipschitz constant of $h_i$ for $i=1,2$;
; when the $h_i$ are  differentiable, we have $Y_i= h_i'(F_i)$, $i=1,2$ (see, for instance, \cite[Proposition 1.2.4]{Nualart}).

We denote by $\delta$ the adjoint of  $D$ given by the duality formula 
\begin{equation} \label{IBP}
\E[\delta(u) F] = \E[ \langle u, DF \rangle_\mathfrak{H}]
\end{equation}
for any $F \in \mathbb{D}^{1,2}$ and $u\in{\rm Dom} \, \delta \subset L^2(\Omega; \HH)$,  the domain of $\delta$. The operator $\delta$ is
also called the Skorohod integral and in the case of the Brownian motion, it coincides
with an extension of the It\^o integral introduced by Skorohod (see \emph{e.g.} \cite{GT, NuPa}).  
In our context, the Dalang-Walsh integral coincides with the Skorohod integral: Any   adapted random field $\Phi$ that satisfies $\E\big[ \|\Phi\|_\H^2 \big]<\infty $ belongs to the domain of $\delta$  and
\[
\delta (\Phi) = 
\int_0^\infty \int_{\R^2} \Phi(s,y) W(d s, d y).
\]
The proof of this result is analogous to the case of  integrals with respect to the Brownian motion  (see \cite[Proposition 1.3.11]{Nualart}), by just replacing real   processes by  $\H_0$-valued processes, where $\H_0$ is defined in \eqref{H_0}.
As a consequence,  the     equation   \eqref{mild} can    be written as 
\begin{equation*}
u(t,x) = 1 +   \delta\big( G_{t-\bullet}(x-\ast) \sigma (  u(\bullet, \ast)  )\big).
\end{equation*}

The operators $D$ and $\delta$ satisfy the commutation relation 
\begin{equation} \label{COMM}
[D, \delta]V: =(D\delta - \delta D)(V) = V.
\end{equation}

By Fubini's theorem and the duality formula \eqref{IBP}, we can interchange the Skorohod integral and Lebesgue integral: Suppose $f_x\in\text{Dom}\delta$ is adapted for each $x$ in some finite measure space $(E, \mu)$ such that $\int_E f_x\mu(dx)$ also belongs to $\text{Dom}\delta$ and $\E\int_E \| f_x\|_\H^2\mu(dx)<\infty$, then 
\begin{align}\label{SF}
\delta\left( \int_E f_x\mu(dx) \right) = \int_E \delta(f_x)\mu(dx)~\text{almost surely}.
\end{align}
Indeed, for any $F\in\mathcal{S}$, 
 \begin{align*}
 \E\left[ F\delta\left( \int_E f_x\mu(dx) \right) \right] &= \E \big\langle DF,  \int_E f_x\mu(dx)  \big\rangle_\H = \int_E \E \big\langle DF, f_x\big\rangle_\H \mu(dx) \\
& = \int_E \E\big[ F \delta(f_x)\big] \mu(dx)  =  \E\left[ F \int_E \delta(f_x) \mu(dx)  \right], 
 \end{align*}
 which gives us \eqref{SF}.  In particular, the equalities in \eqref{sFub} are valid. 
 
With the help of the derivative operator, we can represent $F\in\mathbb{D}^{1,2}$ as a stochastic integral. This is the content of the following two-parameter Clark-Ocone formula, see \emph{e.g.} \cite[Proposition 6.3]{CKNP19} for a proof.

 \begin{lemma}[Clark-Ocone formula] \label{CO}    Given $F\in\mathbb{D}^{1,2}$, we have almost surely
 \[
 F = \E[ F] + \int_{\R_+\times\R^2} \E\big[ D_{s,y} F \vert \F_s \big] W(ds,dy).
 \]
 \end{lemma}
 We end this section with the following useful fact:  
 If $\big\{\Phi_s: s\in\R_+\big\}$ is a jointly measurable and  integrable process satisfying  $\int_{\R_+} \big(\text{Var} ( \Phi_s)\big)^{1/2} ds<\infty$, then
\begin{align}\label{VARF}  
  \sqrt{ \text{Var} \left(  \int_{\R_+} \Phi_sds \right) } \leq   \int_{\R_+} \sqrt{\text{Var} ( \Phi_s)} ds.
  \end{align}

    \section{Gaussian fluctuation of the spatial averages}
  We follow the three steps described in our introduction.  
    
 \subsection{Limiting covariance structure} \label{sub41}
\begin{proposition}\label{FRCOV1}  Suppose $t_1, t_2\in(0,\infty)$. We have, with $\xi(s) = \E\big[ \sigma( u(s,0)) \big]$,
 \begin{align}\label{FRCOV}
 \frac{  \E\big[ F_R(t_1) F_R(t_2) \big] }{R^{4-\beta} } \xrightarrow{R\to\infty} 4\pi^2 c_\beta \kappa_\beta \int_0^{t_1\wedge t_2}   (t_1-s)(t_2-s) \xi^2(s) ds
 \end{align}
 with $\kappa_\beta= \int_{\R^2}d\xi \| \xi\|^{\beta-4} J_1(\| \xi \|)^2\in(0,\infty)$. In particular, for any $t>0$,
\[
{\rm Var}\big(F_R(t)\big) R^{\beta-4} \xrightarrow{R\to\infty} 4\pi^2 c_\beta \kappa_\beta \int_0^{t}  (t-s)^2   \xi^2(s) ds.
\]
\end{proposition}

\begin{proof} Recall that $  F_R(t)  = \int_0^t \int_{\R^2} \varphi_{t,R}(s,y) \sigma ( u(s,y)  ) W(ds,dy)$. Then,  by It\^o's isometry,  
 \[
 \E\big[ F_R(t_1) F_R(t_2) \big]=  \int_0^{t_1\wedge t_2}  \int_{\R^4} \varphi_{t_1,R}(s,y) \varphi_{t_2,R}(s,z) \| y - z\|^{-\beta}  \E\big[  \sigma ( u(s,y)  ) \sigma ( u(s,z)  )\big] dydz ds.
 \]
 We claim that,  as $R\to\infty$,
 \begin{align}\label{weclaim}
  R^{\beta-4}  \int_0^{t_1\wedge t_2}  \int_{\R^4} \varphi_{t_1,R}(s,y) \varphi_{t_2,R}(s,z) \| y - z\|^{-\beta}  \text{Cov}\big[  \sigma ( u(s,y)  ),  \sigma ( u(s,z)  )\big] dydz ds \to 0.
 \end{align}
 Assuming \eqref{weclaim}, we can deduce from Lemma \ref{lemfact}, the stationarity of the process $\{u(t,x): x\in \R^2\}$ and dominated convergence  that
 \[
 \lim_{R\to\infty} \frac{  \E\big[ F_R(t_1) F_R(t_2) \big] }{R^{4-\beta} }  = \lim_{R\to\infty}  \int_0^{t_1\wedge t_2} \xi^2(s) \Psi_R(t_1, t_2;s)ds = \text{RHS of \eqref{FRCOV}},
 \]
 where $\xi(s) = \E[ \sigma(u(s,0)) ]$ is uniformly bounded over $s\in[0, t_1\wedge t_2]$.
 
 \medskip
 
We need to prove \eqref{weclaim} now and it is enough to show for any $s\in(0,t_1\wedge t_2]$
 \begin{align} \label{enough1}
\lim_{\| y -z \| \to\infty}   \text{\rm Cov}\big[ \sigma ( u(s,y)  ) ,  \sigma ( u(s,z)  )\big] =0.
 \end{align}
Indeed,  if \eqref{enough1} holds for any given $s\in(0, t_1\wedge t_2]$, then for arbitrarily small $\e>0$, there is some $K = K(\e,s)$ such that $  \text{Cov}\big[ \sigma ( u(s,y)  ) ,  \sigma ( u(s,z)  )\big]  < \e$, for $\|  y - z \| \geq K$. By Lemma \ref{lemfact}, we deduce 
 \begin{align*}
 &R^{\beta-4}    \int_{\|  y - z \| \geq K} \varphi_{t,R}(s,y) \varphi_{t,R}(s,z) \| y - z\|^{-\beta}  \text{\rm Cov}\big[ \sigma ( u(s,y)  ) ,  \sigma ( u(s,,z)  )\big] dydz  \\
 & \leq  \e \Psi_R(t_1, t_2; s) \lesssim \e,
 \end{align*}
 while  using the  uniform $L^2$-boundedness of $u(t,x)$, we get
  \begin{align*}
 &R^{\beta-4}    \int_{\|  y - z \| < K} \varphi_{t,R}(s,y) \varphi_{t,R}(s,z) \| y - z\|^{-\beta}  \text{\rm Cov}\big[ \sigma ( u(s,y)  ) ,  \sigma ( u(s,z)  )\big] dydz  \\
 &\lesssim   R^{\beta-4}   \int_{ \|  y - z \| < K } \varphi_{t,R}(s,y) \varphi_{t,R}(s,z) \| y - z\|^{-\beta}  dydz \\
 & = R^{\beta-4}  \int_{\| \xi \| < K}  d\xi \| \xi \|^{-\beta} \left(  \int_{\R^2} \varphi_{t,R}(s, z+\xi) \varphi_{t,R}(s, z) dz  \right) \lesssim     R^{\beta-2}  \int_{\| \xi \| < K}  d\xi \| \xi \|^{-\beta} ~\text{by \eqref{fact00}} \\
 &\lesssim R^{\beta-2} \xrightarrow{R\to\infty} 0.
 \end{align*}
That is, we just proved for any $s\in(0, t_1\wedge t_2]$,
\[
R^{\beta-4}    \int_{\R^4} \varphi_{t,R}(s,y) \varphi_{t,R}(s,z) \| y - z\|^{-\beta}  \text{\rm Cov}\big[ \sigma ( u(s,y)  ) ,  \sigma ( u(s,z)  )\big] dydz\xrightarrow{R\to\infty}0,
\]
where the LHS is uniformly bounded in $R>0$ and $s\in(0,t_1\wedge t_2]$ in view of Lemma \ref{lemfact}. Then the claim \eqref{weclaim} follows from the dominated convergence.

\medskip

It remains to verify  \eqref{enough1}. By  Theorem \ref{THM}, for any $0<s < t$,
 \[
 \| D_{s,y} u(t,x) \|_p \lesssim  G_{t-s}(x-y).
 \]
   By Lemma \ref{CO},
 \[
  \sigma  ( u(s,y)  ) = \E\big[   \sigma  ( u(s,y)  )  \big] + \int_0^s \int_{\R^2} \E\Big[ D_{r,\gamma}\big(  \sigma(  u(s,y)    )  \big) \vert \mathscr{F}_r \Big] W(dr, d\gamma).
 \]
 As a consequence,
\begin{equation*}
 \E\big[ \sigma  (u(s,y) )\sigma  (u(s,z) ) \big]  = \xi^2(s)+ T(s, y,z),
 \end{equation*}
 where
\begin{align}
T(s, y,z)  =\int_0^s    \int_{\R^4}   \E\Big(  \E\big[  D_{r,\gamma} ( \sigma (u(s,y) )  ) | \mathcal{F}_r\big]  \E\big[ D_{r,\gamma'}\big( \sigma (u(s,z)) \big) | \mathcal{F}_r\big]  \Big)  \| \gamma-\gamma' \|^{-\beta} d\gamma d\gamma' dr.  \notag 
\end{align}
By the chain-rule \eqref{cr} for the derivative operator,  
$$
 D_{r,\gamma}\big( \sigma (u(s,y)) \big) = \Sigma_{s,y} D_{r,\gamma} u(s,y)
 $$
 with $\Sigma_{s,y}$  an adapted random field  uniformly bounded by   $L$, where we recall  that $L$ is the Lipschitz constant of $\sigma$. This implies,  
 \begin{align} \notag
   \Big\vert  \E\Big(  \E\big[  D_{r,\gamma} ( \sigma (u(s,y) )  ) | \mathcal{F}_r\big]  \E\big[ D_{r,\gamma'}\big( \sigma (u(s,z)) \big) | \mathcal{F}_r\big]  \Big)  \Big\vert &  \lesssim \big\|D_{r,\gamma} u(s,y)\big\| _2   \big\|D_{r,\gamma'} u(s,z)\big\|_2 \notag \\
   & \lesssim  G_{s-r}(\gamma-y) G_{s-r}(\gamma'-z). \notag 
 \end{align}
   Thus,
 \begin{align*}
\vert  T(s, y,z) \vert &\lesssim \int_0^s \int_{\R^4}  G_{s-r}(\gamma-y) G_{s-r}(\gamma'-z)  \| \gamma- \gamma' \|^{-\beta} d\gamma d\gamma' dr.
 \end{align*}
 Suppose $\| y -z \| > 2s$, then 
 $$G_{s-r}(\gamma-y) G_{s-r}(\gamma'-z)  \| \gamma- \gamma' \|^{-\beta} \leq  G_{s-r}(\gamma-y) G_{s-r}(\gamma'-z) \big(  \| y-z \| - 2s \big)^{-\beta}$$
 from which we get
  \begin{align*}
\vert  T(s, y,z) \vert &\lesssim \big(  \| y-z \| - 2s \big)^{-\beta}  \int_0^s \int_{\R^4}  G_{s-r}(\gamma-y) G_{s-r}(\gamma'-z)  d\gamma d\gamma' dr    \xrightarrow{\| y-z\| \to\infty}0.
 \end{align*}
 This implies  \eqref{enough1} and hence concludes our proof.  \qedhere
 
  \end{proof}

  \subsection{Convergence of finite-dimensional distributions}

As  it was  explained in the introduction, a basic ingredient for the  convergence of finite-dimensional distributions is
the following estimate
 \begin{align}\label{finalbdd}
       R^{2\beta-8} {\rm Var}\big(   \langle DF_R(t_1), V_{t_2,R}  \rangle_\H \big) \lesssim R^{-\beta} ~\text{for $R\geq t_1+t_2$},            
 \end{align}
 where we recall that  $V_{t,R}(s,y) =\varphi_{t,R}(s,y) \sigma(u(s,y))$ and $\varphi_{t,R}$ is defined in \eqref{varphi}.
 
 \medskip
    Note that  the Malliavin-Stein bound \eqref{1DNP} and the above bound \eqref{finalbdd} with $t_1=t_2=t$ lead to the quantitative CLT in  \eqref{QCLT}. In fact, from \eqref{finalbdd} and \eqref{1DNP}, we have for any fixed $t>0$ and $Z\sim \mathcal{N}(0,1)$,
 \[
 d_{\rm TV}\big( F_R(t)/\sigma_R, Z \big) \leq \frac{2}{\sigma_R^2} \sqrt{ {\rm Var}\big(   \langle DF_R(t), V_{t,R}  \rangle_\H \big) }\lesssim
\frac{1}{\sigma_R^2} R^{4-\frac {3\beta}2}, \,\, R\ge 2t ;
 \]
  by Proposition \ref{FRCOV1}, $\sigma_R^2 R^{\beta-4}$ converges to some explicit positive constant, see \eqref{FRCOV}. So we can write, for all $R\ge R_t$
  \[
   d_{\rm TV}\big( F_R(t)/\sigma_R, Z \big) \leq C R^{-\beta/2}, 
  \]
 where $R_t$ is some  constant that does not depend on $R$. As the total variation distance is aways bounded by $1$, we can write for $R\leq R_t$,
 \[
   d_{\rm TV}\big( F_R(t)/\sigma_R, Z \big) \leq  1 \leq (R_t)^{\beta/2} R^{-\beta/2}, \forall R\leq R_t.
   \]
 Therefore, the bound \eqref{QCLT} follows.
 
 \medskip

Note  that  \eqref{finalbdd}, together with Proposition \ref{lem612}, implies the convergence in law of the finite dimensional distributions. In fact, fix any integer   $m\geq 1$ and choose $m$ points $t_1,\ldots,t_m\in(0\,,\infty)$, then
	 consider the  random vector $\Phi_R= \big(F_R(t_1), \dots, F_R(t_m) \big)$  and	  let $\mathbf{G}=(\mathcal{G}_1 \,,\ldots,\mathcal{G}_m)$ denote a centered Gaussian random vector
	with covariance matrix $(C_{i,j})_{1\le i,j\le m}$ given by
	\[
C_{i,j} :=	  4\pi^2 c_\beta \kappa_\beta \int_0^{t_i\wedge t_j} (t_i-s)(t_j-s) \xi^2(s)ds.    
	\]
	Recall from \eqref{sFub} that  
	$F_R(t_i) =\delta( V_{t_i,R})$ for all $i=1,\ldots,m$. 	Then,  by  \eqref{MS612}  we can write
		\begin{equation} \label{equa7}
		\big\vert  \E( h(R^{\frac \beta 2-2}\Phi_R)) -\E (h(\mathbf{G})) \big\vert  \leq \frac{m}{2} \|h ''\|_\infty 
		\sqrt{   \sum_{i,j=1}^m   \E \left( \left|
		 C_{i,j} - R^{\beta-4} \langle  DF_R(t_i)  \,, V_{t_j, R} \rangle_{\HH} \right|^2
		\right)}
	\end{equation}
	for every $h\in C^2(\R^m)$ with bounded second partial derivatives. Thus,  in view of \eqref{equa7},  in order to show the  convergence in law of $R^{\frac \beta 2-2} \Phi_R$ to $\mathbf{G}$, it suffices to show that for any $i,j =1,\dots, m$, 
	\begin{equation} \label{h6}
	\lim_{R\rightarrow \infty}  \E \left( \left\vert   C_{i,j} - R^{\beta-4} \langle  DF_R(t_i)  \,, V_{t_j, R} \rangle_{\HH} \right\vert^2
		\right)=0.
		\end{equation}
	Notice that, by the duality relation  \eqref{IBP} and the convergence \eqref{FRCOV},   we have
	\begin{align} 
	R^{\beta-4}\E \Big( \big\langle DF_R(t_i) \,,V_{t_j,R}   \big\rangle_{\HH} \Big) &=  R^{\beta-4}	\E   \big[   F_R(t_i)  \delta( V_{t_j,R} )  \big] \notag\\
	& =  R^{\beta-4} 	\E   \big[   F_R(t_i)  F_R(t_j)   \big]   \xrightarrow{R\to\infty} C_{i,j}. \label{h7}
	\end{align}
	Therefore, the convergence \eqref{h6} follows immediately from  \eqref{h7} and  \eqref{finalbdd}.
		Hence the finite-dimensional distributions of
	$ \{R^{\frac \beta 2-2} F_R(t): t\in\R_+\}$ converge to those of
	$\mathcal{G}$ as $R\to\infty$.

 \medskip
 The rest of this subsection is then devoted to the proof of \eqref{finalbdd}.  
 
 \medskip
 
 \begin{proof}[Proof of \eqref{finalbdd}]
  Recall from \eqref{sFub} that
 $$ F_R(t) = \int_{B_R} (u(t,x) -1)dx = \delta(V_{t,R}) \quad{\rm with} \quad
V_{t,R}(s,y) = \varphi_{t,R}(s,y) \sigma(u(s,y)).
$$
The commutation relation  \eqref{COMM} implies for $s\le t$, 
\begin{equation} \label{EC2}
D_{s,y} F_R(t) = D_{s,y} \delta(V_{t,R})= V_{t,R}(s,y) + \delta (D_{s,y} V_{t,R}).
\end{equation}
By the chain rule for the derivative operator (see \eqref{cr})
\begin{equation} \label{EC1}
D_{s,y}[ V_{t,R}(r,z)] =\varphi_{t,R}(r,z) D[ \sigma(u(r,z))] =\varphi_{t,R} (r,z) \Sigma_{r,z} D_{s,y} u(r,z),
\end{equation}
where  $\Sigma_{r,z}$ is an adapted random field bounded by the Lipschitz constant of $\sigma$.
  Substituting \eqref{EC1} into \eqref{EC2}, yields, for $s\le t$,
\[
D_{s,y}F_R(t) =   \varphi_{t,R}(s,y) \sigma(u(s,y)) +  \int_s^t\int_{\R^2} \varphi_{t,R}(r,z) \Sigma_{r,z} D_{s,y} u(r,z) W(dr, dz).
\]
Then, for $t_1, t_2\in(0,\infty)$, we can write 
$\big\langle DF_R(t_1), V_{t_2, R} \big\rangle_\H = A_1 + A_2$, with 
\begin{align*}
A_1 &= \big\langle  V_{t_1, R}, V_{t_2, R} \big\rangle_\H =\int_0^{t_1\wedge t_2} \int_{ \R^4} \varphi_{t_1, R}(s,y) \varphi_{t_2, R}(s,z) \sigma( u(s,y)) \sigma( u(s,z)) \| y-z\|^{-\beta}dydzds  
\end{align*}
and
\begin{align*}
A_2 &=\int_0 ^{t_1\wedge t_2}\int_{\R^4}  \left( \int_s^{t_1}\int_{\R^2} \varphi_{t_1, R}(r,z) \Sigma_{r,z}  D_{s,y}u(r,z)  W(dr,dz)\right) \\
& \qquad\qquad \times  \| y -y'\|^{-\beta} V_{t_2,R}(s,y')dsdydy'.  
\end{align*}

\medskip
\noindent
{\it {\rm (i)} Estimation  of ${\rm Var}(A_1)$}. From  \eqref{VARF}, we deduce that $\text{Var}(A_1)$ is bounded by
\begin{align}\label{VAR-term}
\left(    \int_{0}^{t_2\wedge t_1} \left(   \text{Var}   \int_{\R^4} \varphi_{t_1, R}(s,y) \varphi_{t_2, R}(s,z) \sigma( u(s,y)) \sigma( u(s,z)) \| y-z\|^{-\beta}dydz  \right)^{1/2}ds\right)^2.
\end{align}
Note that the variance term  in \eqref{VAR-term} is equal to 
\begin{align}
&  \int_{\R^8} \varphi_{t_1, R}(s,y) \varphi_{t_2, R}(s,z) \varphi_{t_1, R}(s,y') \varphi_{t_2, R}(s,z')   \| y-z\|^{-\beta} \| y'-z'\|^{-\beta} \notag  \\
 &\qquad\qquad \times  \text{Cov}\Big[  \sigma( u(s,y)) \sigma( u(s,z))  ,  \sigma( u(s,y')) \sigma( u(s,z'))      \Big]  dydz dy'dz' . \label{VARterm}
  \end{align}
To estimate the covariance term, we apply the Clark-Ocone formula (see Lemma \ref{CO}) to write 
\begin{align*}
 & \sigma( u(s,y)) \sigma( u(s,z)) -\E[   \sigma( u(s,y)) \sigma( u(s,z)) ] \\
 & \qquad = \int_0^s \int_{\R^2} \E\Big\{ D_{r,\gamma} \big( \sigma( u(s,y)) \sigma( u(s,z)) \big) \vert \F_r \Big\} W(dr, d\gamma).
\end{align*}
Then we apply It\^o's isometry to obtain
\begin{align}
&\quad\text{Cov}\Big[  \sigma( u(s,y)) \sigma( u(s,z))  ,  \sigma( u(s,y')) \sigma( u(s,z'))      \Big]  \label{COV-term}    \\
&= \int_0^s \int_{\R^4} \E\Bigg[  \E\big\{ D_{r,\gamma} \big( \sigma( u(s,y)) \sigma( u(s,z)) \big) \vert \F_r \big\} \E\big\{ D_{r,\gamma'} \big( \sigma( u(s,y')) \sigma( u(s,z')) \big) \vert \F_r \big\} \Bigg]  \notag \\
 & \qquad \times \| \gamma - \gamma'\|^{-\beta}d\gamma d\gamma' dr, \notag
\end{align}
where, by the chain rule \eqref{cr},  
$$D_{r,\gamma} \big( \sigma( u(s,y)) \sigma( u(s,z)) \big) =  \sigma( u(s,y))  \Sigma_{s,z} D_{r,\gamma}  u(s,z)+ \sigma( u(s,z))  \Sigma_{s,y} D_{r,\gamma}  u(s,y).
$$
 Then by Cauchy-Schwarz inequality and Theorem \ref{THM}, we can see that the above covariance term \eqref{COV-term} is bounded by 
\begin{align*}
&\quad \int_0^s \int_{\R^4} \Big\|  D_{r,\gamma} \big( \sigma( u(s,y)) \sigma(u(s,z)) \big)  \Big\|_2  \Big\|  D_{r,\gamma'} \big( \sigma( u(s,y')) \sigma( u(s,z')) \big) \Big\|_2 \| \gamma - \gamma'\|^{-\beta}d\gamma d\gamma' dr \\
&\lesssim  \int_0^sdr \int_{\R^4}  d\gamma d\gamma'    \| \gamma - \gamma'\|^{-\beta}  \Big(   \|  D_{r,\gamma}  u(s,z)  \|_4  +   \|  D_{r,\gamma}  u(s,y)  \|_4\Big)  \\
&\qquad\qquad\qquad \times  \Big(   \|  D_{r,\gamma'}  u(s,z')  \|_4  +   \|  D_{r,\gamma'}  u(s,y')  \|_4\Big)      \\
&\lesssim   \int_0^sdr \int_{\R^4}  d\gamma d\gamma'    \| \gamma - \gamma'\|^{-\beta}  \big(  G_{s-r}(z-\gamma)  +  G_{s-r}(y-\gamma)  \big)  \big(   G_{s-r}(z'-\gamma') +    G_{s-r}(y'-\gamma')  \big).
\end{align*}
Now we can plug the last estimate into   \eqref{VARterm}   for further computations:
\begin{align}
&\quad  \text{Var}  \left( \int_{\R^4} \varphi_{t_1, R}(s,y) \varphi_{t_2, R}(s,z) \sigma( u(s,y)) \sigma( u(s,z)) \| y-z\|^{-\beta}dydz\right)\notag \\
&\lesssim    \int_0^sdr  \int_{\R^{12}} \varphi_{t_1, R}(s,y) \varphi_{t_2, R}(s,z) \varphi_{t_1, R}(s,y') \varphi_{t_2, R}(s,z')   \| y-z\|^{-\beta} \| y'-z'\|^{-\beta}     \| \gamma - \gamma'\|^{-\beta} \notag \\
&   \times \big(  G_{s-r}(z-\gamma)  +  G_{s-r}(y-\gamma)  \big)  \big(   G_{s-r}(z'-\gamma') +    G_{s-r}(y'-\gamma')  \big)  d\gamma d\gamma'  dydz dy'dz'. \label{from}
\end{align}
In order to obtain  $\text{Var}(A_1)\lesssim R^{8-3\beta}$,  it is enough to show  $\sup_{s\leq t_1\wedge t_2} \mathcal{T}_s\lesssim R^{8-3\beta} $ with
\begin{align*}
 \mathcal{T}_s:&=  \int_0^sdr  \int_{\R^{12}} \varphi_{t_1, R}(s,y) \varphi_{t_2, R}(s,z) \varphi_{t_1, R}(s,y') \varphi_{t_2, R}(s,z')   \| y-z\|^{-\beta} \| y'-z'\|^{-\beta}    \\
&  \quad \times  \| \gamma - \gamma'\|^{-\beta}G_{s-r}(z-\gamma)    G_{s-r}(z'-\gamma')  d\gamma d\gamma'  dydz dy'dz'  
\end{align*}
 as other terms from \eqref{from} can be estimated {\it in the same way} with  {\it the same bound}.

   For $s\in(0, t_1\wedge t_2]$, we write, using \eqref{varphi}, 
 \begin{align*}
 \mathcal{T}_s &=   \int_0^s dr  \int_{B_R^4} \int_{\R^{12}}  G_{t_1-s}(x_1-y) G_{t_1-s}(x'_1-y')   G_{t_2-s}(x_2-z)G_{t_2-s}(x'_2-z')G_{s-r}(z-\gamma)         \\
&   \times G_{s-r}(z'-\gamma')     \| \gamma - \gamma'\|^{-\beta} \| y-z\|^{-\beta} \| y'-z'\|^{-\beta}  d\gamma d\gamma'  dydz dy'dz' dx_1dx_1'dx_2dx_2'.
 \end{align*}
 Making the change of variables
 \[
 ( \gamma, \gamma', y, z, y', z', x_1, x_1', x_2, x_2')\to R ( \gamma, \gamma', y, z, y', z', x_1, x_1', x_2, x_2')
 \] 
 and using $G_t(Rz) = R^{-1} G_{tR^{-1}}(z)$ for every $t, R>0$ yields
\begin{align*} 
&R^{-14+3\beta}  \mathcal{T}_s =  \int_0^s dr  \int_{B_1^4} \int_{\R^{12}}  G_{\frac{t_1-s}{R}}(x_1-y) G_{\frac{t_1-s}{R}}(x'_1-y')   G_{\frac{t_2-s}{R}}(x_2-z)G_{\frac{t_2-s}{R}}(x'_2-z')         \\
&   \times G_{\frac{s-r}{R}}(z-\gamma)  G_{\frac{s-r}{R}}(z'-\gamma')     \| \gamma - \gamma'\|^{-\beta} \| y-z\|^{-\beta} \| y'-z'\|^{-\beta}  d\gamma d\gamma'  dydz dy'dz' dx_1dx_1'dx_2dx_2'.
 \end{align*}
Using the fact  \eqref{fact01}, we can integrate out $x_1,x_1',x_2,x_2'$ to bound $R^{-14+3\beta}  \mathcal{T}_s$ by
\begin{align}
& R^{-10+3\beta} (t_1-s)^2 (t_2-s)^2 \int_0^s dr   \int_{\R^{12}}      \1_{\{ \| y\| \vee \| y'\| \vee \|z\| \vee \| z'\| \vee \| \gamma\| \vee \| \gamma' \| \leq  1+ (t_1+t_2)R^{-1}   \}}   \notag   \\
&   \qquad \times G_{\frac{s-r}{R}}(z-\gamma)  G_{\frac{s-r}{R}}(z'-\gamma')     \| \gamma - \gamma'\|^{-\beta} \| y-z\|^{-\beta} \| y'-z'\|^{-\beta}  d\gamma d\gamma'  dydz dy'dz' .  \label{dydy}
\end{align}
Suppose $R\geq t_1+t_2$ and notice that
\[
\sup_{z\in B_2} \int_{B_2} \| y -z\|^{-\beta}dy \leq \int_{B_4} \| y \|^{-\beta}dy = \frac{2\pi}{2-\beta} 4^{2-\beta} <\infty.
\]
Therefore, integrating   out $y, y'$ in \eqref{dydy}, we obtain
 \begin{align*}
\mathcal{T}_s \lesssim &R^{10-3\beta}  \int_0^s dr   \int_{\R^{8}}      \1_{\{   \|z\| \vee \| z'\| \vee \| \gamma\| \vee \| \gamma' \| \leq 2 \}}      G_{\frac{s-r}{R}}(z-\gamma)  G_{\frac{s-r}{R}}(z'-\gamma')     \| \gamma - \gamma'\|^{-\beta}   d\gamma d\gamma'   dz dz'.  
\end{align*}
We further integrate out $z, z'$ and use \eqref{fact01} again to write 
 \begin{align*}
\sup_{s\leq t_1\wedge t_2}\mathcal{T}_s \lesssim &R^{8-3\beta}     \int_{\R^{8}}      \1_{\{   \| \gamma\| \vee \| \gamma' \| \leq 2 \}}          \| \gamma - \gamma'\|^{-\beta}   d\gamma d\gamma'  \lesssim R^{8-3\beta}.
\end{align*}
So we obtain  $\text{Var}(A_1) \lesssim R^{8-3\beta}$ for $R\geq t_1+t_2$, where the implicit constant does not depend on $R$.

\bigskip
Next we estimate the variance of $A_2$.

\medskip
{\it  {\rm (ii)} Estimate of ${\rm Var}(A_2)$}.  Using again \eqref{VARF}, we write
\begin{align*}
\text{Var}(A_2) &\leq \Bigg( \int_{0}^{t_1\wedge t_2} \Bigg\{ \text{Var} \int_{\R^4}  \left( \int_s^{t_1}\int_{\R^2} \varphi_{t_1, R}(r,z) \Sigma_{r,z}  D_{s,y}u(r,z)  W(dr,dz)\right)  \| y -y'\|^{-\beta} \\
&\qquad\qquad \qquad \times \varphi_{t_2, R}(s,y') \sigma(u(s,y'))  dydy' \Bigg\}^{1/2} ds  \Bigg)^{2} =: \left(  \int_{0}^{t_1\wedge t_2} \sqrt{ \mathcal{U}_s} ds \right)^2.
\end{align*}
As before, we will show $\sup_{s\leq t_2\wedge t_1}\mathcal{U}_s\lesssim R^{8-3\beta}$.

 First note that   
 \[
 \int_s^{t_1}\int_{\R^2} \varphi_{t_1, R}(r,z) \Sigma_{r,z}  D_{s,y}u(r,z)  W(dr,dz) = \mathfrak{M}_{s,y}(t_1),
 \]
  where $\big\{\mathfrak{M}_{s,y}(\tau): \tau\in[s, t_1] \big\}$ is the square-integrable martingale given by
\[
\mathfrak{M}_{s,y}(\tau):=\int_s^{\tau}\int_{\R^2} \varphi_{t_1, R}(r,z) \Sigma_{r,z}  D_{s,y}u(r,z)  W(dr,dz).
\]
Then we deduce from the martingale property that
\[
\E\big[ \sigma(u(s,y')) \mathfrak{M}_{s,y}(t_1) \big] = \E\big[ \sigma(u(s,y'))  \E( \mathfrak{M}_{s,y}(t_1) \vert \F_s) \big] = 0,
\]
that is, $\mathfrak{M}(t_1)$ and $\sigma\big(u(s,y')\big) $ are uncorrelated.  Moreover, by It\^o's formula,
\[
\mathfrak{M}_{s,y}(t_1) \mathfrak{M}_{s,  \wt{y}  }(t_1) =\underbrace{ \int_s^{t_1} \mathfrak{M}_{s,y}(\tau)d\mathfrak{M}_{s,\wt{y} }(\tau) +\int_s^{t_1} \mathfrak{M}_{s,\wt{y} }(\tau)d\mathfrak{M}_{s,y}(\tau) }_{\rm martingale-part}+ \langle  \mathfrak{M}_{s,y}, \mathfrak{M}_{s,\wt{y} } \rangle_{t_1},
\]
where the bracket $\langle  \mathfrak{M}_{s,y}, \mathfrak{M}_{s,\wt{y} } \rangle_{t_1}$ between both martingales  is equal to
\[
\int_s^{t_1} \int_{\R^4} \varphi_{t_1, R}(r,z) \Sigma_{r,z}  \big(D_{s,y}u(r,z)\big)\varphi_{t_1, R}(r,\wt{z}) \Sigma_{r,\wt{z}}  \big(D_{s,\wt{y} }u(r,\wt{z})\big) \|z-\wt{z}\|^{-\beta}dzd\wt{z} dr.
\]
So, using the estimate  \eqref{IMP}, we obtain
\begin{align*}
&\quad  \E\Big[ \mathfrak{M}_{s,y}(t_1) \mathfrak{M}_{s,\wt{y} }(t_1)  \sigma(u(s, y')) \sigma(u(s, \wt{y'}))     \Big] \\
&=  \E\Big[ \E\big(\mathfrak{M}_{s,y}(t_1) \mathfrak{M}_{s,\wt{y} }(t_1) \vert \F_s \big) \sigma(u(s, y')) \sigma(u(s, \wt{y'}))     \Big]  \lesssim  \big\| \langle  \mathfrak{M}_{s,y}, \mathfrak{M}_{s,\wt{y} } \rangle_{t_1} \big\|_2\\
& \lesssim \int_s^{t_1} \int_{\R^4} \varphi_{t_1, R}(r,z)    \| D_{s,y}u(r,z)\|_4 \varphi_{t_1, R}(r,\wt{z})  \| D_{s,\wt{y} }u(r,\wt{z})\|_4 \|z-\wt{z}\|^{-\beta}dzd\wt{z} dr \\
&\lesssim \int_s^{t_1} \int_{\R^4} \varphi_{t_1, R}(r,z)   G_{r-s}(y-z)   \varphi_{t_1, R}(r,\wt{z}) G_{r-s}(\wt{y} -\wt{z})   \|z-\wt{z}\|^{-\beta}dzd\wt{z} dr.
\end{align*}
As a consequence,  the variance-term $\mathcal{U}_s$  is indeed a second moment and 
\begin{align*} 
\mathcal{U}_s&= \int_{\R^8} dydy'd\wt{y} d\wt{y'} \| y -y'\|^{-\beta} \| \wt{y} -\wt{y'}\|^{-\beta} \varphi_{t_2, R}(s, y')\varphi_{t_2, R}(s, \wt{y'})   \notag \\
&\qquad\times \E\Big[ \mathfrak{M}_{s,y}(t_1) \mathfrak{M}_{s,y'}(t_1)  \sigma(u(s, y')) \sigma(u(s, \wt{y'}))     \Big]\\
&\lesssim  \int_s^{t_1} dr  \int_{\R^{12}} dzd\wt{z} dydy'd\wt{y} d\wt{y'} \| y -y'\|^{-\beta} \| \wt{y} -\wt{y'}\|^{-\beta} \| z - \wt{z} \|^{-\beta}  \\
&\quad\times \varphi_{t_2, R}(s, y')\varphi_{t_2, R}(s, \wt{y'}) \varphi_{t_1, R}(r,z)\varphi_{t_1, R}(r,\wt{z})  G_{r-s}(y-z)G_{r-s}(\wt{y}-\wt{z}), 
\end{align*}
which has the same kind  of expression as $\mathcal{T}_s$. The same arguments that led to the uniform estimate of $\mathcal{T}_s$ yields 
\[
\sup_{s\leq t_1\wedge t_2} \mathcal{U}_s \lesssim  R^{8-3\beta},
\]
for $R\geq t_1+t_2$, thus we obtain $\text{Var}(A_2) \lesssim R^{8-3\beta}$ for $R\geq t_1+t_2$. Hence, for $R\geq t_1+t_2$,
\[
  R^{2\beta-8} {\rm Var}\big(   \langle DF_R(t_1), V_{t_2,R}  \rangle_\H \big) \lesssim R^{2\beta-8}  \big[ \text{Var}(A_2)  +\text{Var}(A_1) \big]  \lesssim R^{-\beta}.
\]
This completes the proof of  \eqref{finalbdd}. 
 \end{proof}

 \subsection{Tightness}  Set $ q= \frac 2 {4-\beta} \in (1/2,1)$. As explained in the introduction, 
 by the Kolmogorov-Chentsov criterion for tightness, it is enough to prove  the inequality \eqref{KC}:   For any $T>0$, 
  $p\geq 2$ and  for any $0\leq s<t\leq T \leq R$,  
  \begin{align}\label{KOLT}
  \big\| F_R(t) - F_R(s) \big\|_p \lesssim  R^{1/q}\sqrt{t-s},
  \end{align}
  where the implicit constant does not depend on $t,s$ or $R$.
 \begin{proof}[Proof of \eqref{KOLT}]  Recall that
 $
 F_R(t)  = \int_{0}^t\int_{\R^2} \varphi_{t,R}(s,y) \sigma(u(s,y)) W(ds,dy)
 $. Then by BDG inequality \eqref{BDG} and \eqref{Sobolev}   we have, with the convention that $ \varphi_{s,R}(r,y)=0$ if $ r>s$,
 \begin{align*}
   \big\| F_R(t) - F_R(s) \big\|_p^2 &\lesssim
    \Bigg\| \int_{ [0,t]\times \R^4}    ( \varphi_{t,R}(r,y) -\varphi_{s,R}(r,y)\big)  \sigma(u(r,y))        
    ( \varphi_{t,R}(r,z) -\varphi_{s,R}(r,z)\big) \\
    & \qquad \times  \sigma(u(r,z)) \| y-z \|^{-\beta} dydzdr \Bigg\|_{p/2} \\
     &\lesssim  \left\| \int_{0}^t  dr\left(\int_{\R^2} \Big\vert \big( \varphi_{t,R}(r,y) -\varphi_{s,R}(r,y)\big)  \sigma(u(r,y))   \Big\vert^{2q} dy\right)^{1/q}  \right\|_{p/2}.
 \end{align*}    
    Applying Minkowski's inequality yields
    \begin{align}
     \big\| F_R(t) - F_R(s) \big\|_p^2 &\lesssim
\int_{0}^t  dr\left(\int_{\R^2} \big\vert \varphi_{t,R}(r,y) -\varphi_{s,R}(r,y)  \big\vert^{2q} \|  \sigma(u(r,y)) \|_p^{2q}  dy\right)^{1/q} \notag  \\
   &\lesssim \int_{0}^t dr\left(\int_{\R^2} \big\vert \varphi_{t,R}(r,y) -\varphi_{s,R}(r,y)  \big\vert^{2q}  dy\right)^{1/q}.  \label{APP2}
 \end{align}
Note that
  \begin{align*}
   \big\vert \varphi_{t,R}(r,y) -\varphi_{s,R}(r,y)  \big\vert &= {\bf 1}_{\{ r\geq s \}} \int_{B_R} G_{t-r}(x-y)dx \\
   &\quad +  {\bf 1}_{\{ r< s \}}  \int_{B_R}  {\bf 1}_{\{ \| x-y\| < s-r \}} \big[ G_{s-r}(x-y)-G_{t-r}(x-y) \big]dx \\
   &\quad +  {\bf 1}_{\{ r< s \}} \int_{B_R}  {\bf 1}_{\{ \| x-y\| \geq s-r \}} G_{t-r}(x-y) dx \\
   & =: S_1+S_2+S_3.
  \end{align*}
The first summand $S_1$ is bounded by  $ {\bf 1}_{\{ r\geq s \}} (t-r) {\bf 1}_{\{ \| y\| \leq R+t \}}  \leq (t-s){\bf 1}_{\{ \| y\| \leq R+t \}} $,     in view of Lemma \ref{LFact}-(2). For  the second summand, we can write
 \begin{align*}
 S_2 & \le   {\bf 1}_{\{ r< s \}}   {\bf 1}_{\{ \| y\| \leq R+s \}}    \int_{B_R}  {\bf 1}_{\{ \| x\| < s-r \}} \big[ G_{s-r}(x)-G_{t-r}(x) \big]dx \\
  &\leq {\bf 1}_{\{ r< s \}}   {\bf 1}_{\{ \| y\| \leq R+s \}}  \int_{\{\| x\| <s-r\}} \left(\frac{1}{2\pi    \sqrt{ (s-r)^2 - \|x\|^2       }} -\frac{1}{2\pi    \sqrt{ (t-r)^2 - \|x\|^2       }} \right)dx \\
  &= {\bf 1}_{\{ r< s \}}   {\bf 1}_{\{ \| y\| \leq R+s \}}  \sqrt{t-s} \Big( \sqrt{t+s-2r}  -\sqrt{t-s}\Big) \quad\text{by explicit computation}\\
  &\lesssim \sqrt{t-s} {\bf 1}_{\{ \| y\| \leq R+s \}} ;
 \end{align*}
In the same way, the third summand can be bounded as follows 
 \begin{align*}
S_3 \le  {\bf 1}_{\{ r< s \}}  {\bf 1}_{\{ \| y\| \leq R+t \}}  \int_{\R^2}  {\bf 1}_{\{ s-r \leq \| x\| < t-r \}} G_{t-r}(x) dx \lesssim  {\bf 1}_{\{ \| y\| \leq R+t \}}  \sqrt{t-s}.
 \end{align*}
Therefore,  we can continue with \eqref{APP2} to write 
 \begin{align*}
   \big\| F_R(t) - F_R(s) \big\|_p^2  \lesssim \int_{0}^t dr\left(\int_{\R^2} (t-s)^q   {\bf 1}_{\{ \| y\| \leq R+t \}}  dy\right)^{1/q}\lesssim  (t-s) (R+t)^{2/q}.
 \end{align*}
 This implies \eqref{KOLT}.
 \end{proof}

\section{Fundamental estimate on the Malliavin derivative}
This section is devoted to the proof of  Theorem \ref{THM}. After a useful lemma, we study the convergence and moment estimates for the Picard approximation in Section \ref{Picard}. The main body of the proof of 
 Theorem \ref{THM}  is given in Section \ref{body} and we leave proofs of two  technical lemmas to Section \ref{lemmas}. Recall that $\beta\in(0,2)$ is fixed throughout this paper.

\begin{lemma}\label{fori} Given any random field $\{ \Phi(r,z): (r,z)\in\R_+\times\R^2\}$, 
 we have  for any $x\in\R^2$, $0\leq s<t <\infty$ and $p\geq 2$,
 \begin{align}
& \left\| \int_{s}^t dr \int_{\R^4} dydz\, G_{t-r}(x-y) G_{t-r}(x-z)  \Phi(r,z) \Phi(r,y) \| y-z\|^{-\beta} \right\|_{p/2}  \notag  \\ 
& \le K_\beta t^{\frac{(2-2q)^2}{2q}} \int_s^t  dr   \int_{\R^2} dz\, G_{t-r}^{2q}(x-z)  \big\| \Phi(r,z)   \big\|_p^2, \label{RHS1}
 \end{align}
 where $q=\frac{2}{4-\beta}\in(1/2,1)$ and  the  constant  $K_\beta$  only depends on $\beta$.
 \end{lemma}
\begin{proof} By  \eqref{Sobolev}, there exists some constant $C_\beta$ that only depends on $\beta$ such that 
\begin{align*}
&\quad  \int_{\R^4} dydz\, G_{t-r}(x-y) G_{t-r}(x-z)  \Phi(r,z) \Phi(r, y) \| y-z\|^{-\beta} \\
& \leq C_\beta \left(   \int_{\R^2} dy\, G_{t-r}^{2q}(x-y)  \vert \Phi(r,y) \vert^{2q}   \right)^{1/q} \\
&\leq C_\beta \left(\frac{(2\pi)^{1-2q} }{2-2q} (t-r)^{2-2q} \right)^{\frac{1}{q}-1}  \int_{\R^2} dy\, G_{t-r}^{2q}(x-y)  \vert \Phi(r,y) \vert^2 \\
&\leq K_\beta t^{\frac{(2-2q)^2}{2q}}\int_{\R^2} dy\, G_{t-r}^{2q}(x-y)  \vert \Phi(r,y) \vert^2,
 \end{align*}
where we have used the fact that $G_{t-r}^{2q}(y)dy$,  with $2q<2$, is a finite measure on $\R^2$ with total mass $\frac{(2\pi)^{1-2q} }{2-2q} (t-r)^{2-2q}$ in view of  \eqref{fact01} and we have put $K_\beta=C_\beta  \big( \frac { (2\pi)^{1-2q}}{2-2q} \big)^{\frac 1q-1}$.
 Therefore, a further application of Minkowski's inequality yields  the bound in  \eqref{RHS1}.\qedhere
\end{proof}

 \subsection{Moment estimates for the Picard approximation}  \label{Picard}
 
  Recall the Picard iteration introduced in \eqref{Pi}:  $u_0(t,x)= 1$ and
\begin{equation} \label{picar}
 u_{n+1}(t,x) = 1 + \int_0^t \int_{\R^2} G_{t-s}(x-y) \sigma\big( u_n(s,y) \big) W(ds, dy) ~ \text{for $n\geq 0$}.
\end{equation}
  Using the estimates  \eqref{eq2} and  \eqref{RHS1},
    we can  write with $2q= \frac{4}{4-\beta}\in(1,2)$, $p\ge 2$ and $n\ge 1$,
   \begin{align*}
&  \| u_n(t,x)\|_p^2  \le  2+8p \\
 & \quad \times 
 \left\|  \int_{[0,t]\times\R^4}G_{t-s}(x-z)  G_{t-s}(x-y) \sigma( u_n(s,y) )\sigma( u_n(s,z) ) \|y-z\|^{-\beta}dsdzdy \right\|_{p/2} \\
 & \qquad  \qquad \le 
 2+ 8pK_\beta t^{\frac { (2-2q)^2} { 2q}}  \int_0^t  ds   \int_{\R^2} G_{t-s}^{2q}(x-y)  \| \sigma ( u_{n-1}(s,y) )\|_p^2  dy.
    \end{align*}
 Then,   using   \eqref{fact01}, we can write  
 \begin{align*}
  \| u_n(t,x)\|_p^2  &\le  2+ 8pK_\beta t^{\frac { (2-2q)^2} { 2q}}  \int_0^t  ds   \int_{\R^2} G_{t-s}^{2q}(x-y) 
 \Big( 2 \sigma(0)^2 +2L^2 \| u_{n-1}(s,y) \|_p^2 \Big)  dy\\
 & \le  2+  \frac { 16pK_\beta (2\pi)^{1-2q}  }{ (2-2q) (3-2q)}t^{\frac { (2-2q)^2} { 2q}+3-2q}  \sigma(0)^2\\
 & \qquad  +   16pK_\beta t^{\frac { (2-2q)^2} { 2q}} L^2 \int_0^t  ds   \int_{\R^2} G_{t-s}^{2q}(x-y) 
  \| u_{n-1}(s,y) \|_p^2  dy,
 \end{align*}
where $L$ is the Lipschitz constant of $\sigma$.
 This leads to
 \begin{equation} \label{eq11}
 H_n(t) \le c_1 + c_2 \int_0^t  dsH_{n-1}(s),
 \end{equation}
 where $H_n(t)= \sup_{x\in \R^2} \| u_n(t,x)\|_p^2$,  
 \[
 c_1:= 2+  \frac { pK^*_\beta  \sigma(0)^2 }{ 3-2q}  t^{\frac { (2-2q)^2} { 2q}+3-2q}  
\quad  {\rm
  and
 } \quad
 c_2:=   pK^*_\beta   L^2    t^{\frac { (2-2q)^2} { 2q}+2-2q},
 \]
 where $K^*_\beta= \frac {16 K_\beta   (2\pi)^{1-2q}}{2-2q}= 16 C_\beta \big(  \frac {(2\pi)^{1-2q}}{2-2q}  \big)^{ 1/q} $ is a constant depending only on $\beta$.
 Therefore, by iterating the inequality \eqref{eq11} and taking into account that $H_0(t)=1$, yields
 \[
 H_n(t) \le c_1 \exp(c_2 t).
 \]
 In what follows, we will denote by $C_\beta^*$ a generic constant that only depends on $\beta$ and may be different from line to line. In this way, we obtain  
 \[
  \| u_n(t,x)\|_p \leq \big( \sqrt{2}+ \sqrt{p} C^*_\beta  t^{\frac  {3-\beta}2} \vert \sigma(0) \vert \big) \exp\big(  pC^*_\beta t^{ 2-\beta} L^2\big).
  \]
  As a consequence,
  \begin{equation}  \label{sigmap}
  \| \sigma(u_n(t,x)) \|_p \le | \sigma(0)| + L\big( \sqrt{2}+ \sqrt{p} C^*_\beta  t^{\frac  {3-\beta}2} \vert \sigma(0) \vert \big) \exp\big(  pC^*_\beta t^{2-\beta} L^2\big)     =: \kappa_{p,t}.
  \end{equation}

 \medskip
 
 \subsection{Proof of Theorem \ref{THM}} \label{body}
 The proof will be done in several steps.
 
 \medskip
 \noindent
 {\bf Step 1.} In this step, we will establish the following estimate \eqref{Conq1} for the $p$-norm of the Malliavin derivative of the Picard iteration.

 \begin{proposition}\label{PROP32}
 For any $n\ge 3$
 and any $p\ge 2$
\begin{equation}   \label{Conq1}
  \big\| D_{s,y} u_{n+1}(t,x) \big\|_p \leq   C_{\beta,p, t, L}  \kappa_{p,t} G_{t-s} (x-y),
\end{equation}
 for almost all $(s,y) \in [0,t]\times \R^2$,  where $ \kappa_{p,t} $ is defined in \eqref{sigmap} and the constant $ C_{\beta,p, t, L}$ is given by
\begin{align}
C_{\beta,p, t, L}&:=1+ \sqrt{p}  L C^*_\beta t^{\frac 1q -\frac 12} +  pC_\beta^*L^2  t^{ \frac 2q-1  }+  \sum_{k=3} ^\infty  \frac{ (pC_\beta^*L^2)^{k/2}    }{\sqrt{(k-2)!}}     t^{  k( \frac 1q -\frac 12)},  \label{Cbeta} 
   \end{align}
   with $C^*_\beta$ a constant only depending on $\beta$. 
 \end{proposition}
 
 One key ingredient for proving Proposition \ref{PROP32} is  the following Lemma \ref{LEM2}, which is a consequence of the technical Lemma \ref{LEM1}. Both Lemma \ref{LEM1} and Lemma \ref{LEM2} will be proved in Section \ref{lemmas}.

    \begin{lemma}\label{LEM2} For $q\in(1/2,1)$, $\delta \in [1,1/q]$ and $s<t$,  we have
 \begin{align}
 K_{s,t}(z) := \int_s^t  dr  \big[  G_{t-r}^{2q} \ast G_{r-s}^{2q}(z)  \big]^{\delta}  \lesssim   (t-s)^{1-\delta(2q-1)}   G_{t-s}^{\delta(2q-1)}(z). \label{just0}
 \end{align}
  where the implicit constant only depends on $q$. 
     \end{lemma}

 \begin{proof}[Proof of Proposition \ref{PROP32}] 
Fix  $(t,x) \in \R_+\times \R^2$ and $p\ge 2$. Let us first establish the following weaker estimate:
\begin{align}
&\text{$u_n(t,x)\in\mathbb{D}^{1,p}$ and   $\| D_{s,y} u_n(t,x) \|_p \leq C G_{t-s}(x-y)$,  }  \label{claim:sec42}
\end{align}
for  almost all  $(s,y)\in[0,t]\times\R^2$, where the constant $C$ may depend on $n$. It follows from \eqref{picar} that the claim \eqref{claim:sec42} holds true for $n=0,1$, because $D_{s,y} u_0(t,x)=0$ and $D_{s,y}u_1(t,x) =\sigma(1) G_{t-s}(x-y)$. Now suppose  the claim \eqref{claim:sec42} holds true for $n\geq 1$, then   taking the Malliavin derivative in both sides of equality \eqref{picar} and using the commutation relationship \eqref{COMM} and the chain rule \eqref{cr},  we obtain
 \begin{align*}
 D_{s,y} u_{n+1}(t,x) = G_{t-s}(x-y)  \sigma\big( u_n(s,y) \big) + \int_s^t\int_{\R^2} G_{t-r}(x-z)  \Sigma_{r,z}^{(n)} D_{s,y}u_n(r,z) W(dr,dz),
 \end{align*}
 where  $\big\{\Sigma_{s,y}^{(n)}: (s,y)\in\R_+\times\R^2\big\}$ is an adapted random field that is uniformly bounded by $L$, for each $n$.  We recall that the constant $L$ is the Lipschitz constant of the function $\sigma$ appearing in \eqref{2dSWE}. It follows that
 \begin{align*}
& \big\| D_{s,y} u_{n+1}(t,x)\big\|_p^2  \leq 2 \kappa_{p,t}^2 G^2_{t-s}(x-y)  +  8p \Bigg\|  \int_s^t\int_{\R^4} G_{t-r}(x-z) G_{t-r}(x-z')  \Sigma_{r,z}^{(n)}  \\
&\qquad\qquad\qquad \qquad \times D_{s,y}u_n(r,z) \Sigma_{r,z'}^{(n)} D_{s,y}u_n(r,z') \| z- z' \|^{-\beta} dzdz'dr \Bigg\|_{p/2}~\text{by BDG \eqref{eq2}} \\
&\leq 2 \kappa_{p,t}^2 G^2_{t-s}(x-y) + 8pL^2 C_n^2  \int_s^t \big\| G_{t-r}(x-\bullet)G_{r-s}(y-\bullet)\big\|_{\mathfrak{H}_0}^2
dr,
 \end{align*}
by applying  Minkowski's inequality and using the induction hypothesis,  where $ \kappa_{p,t}$ is defined in \eqref{sigmap}
and $\H_0$ has been introduced in \eqref{H_0}.
  Note that Lemma \ref{HLS} (see \eqref{Sobolev}) implies 
\[
\int_s^t  \big\| G_{t-r}(x-\bullet)G_{r-s}(y-\bullet)\big\|_{\mathfrak{H}_0}^2 \leq  C_\beta \int_s^t dr \big(  G^{2q}_{t-r}\ast G^{2q}_{r-s}\big)^{1/q}(x-y) \leq C^\ast_\beta t^{\frac{1}{q}-1} G_{t-s}^{2-\frac1q}(x-y),
\]
where the last inequality follows from Lemma \ref{LEM2} with $\delta= 1/q$ and $C_\beta^\ast$ is a constant that only depends on  $\beta$. Finally, using 
\begin{equation} \label{EC1}
  G_{t-s}^{2-\frac 1q}(x-y) \leq  \big[2\pi (t-s)\big]^{\frac 1q} G^2_{t-s}(x-y),
 \end{equation}  
  we get  $\| D_{s,y} u_{n+1}(t,x) \|_p \leq C_{n+1} G_{t-s}(x-y)$ with $C_{n+1} = \sqrt{ 2 \kappa_{p,t}^2 +  pL^2 C_n^2 C^\ast_\beta t^{\frac{2}{q}-1}   } $
  and thus by routine  computations, we can show $u_{n+1}(t,x)\in\mathbb{D}^{1,p}$; see also \textbf{Step 2}. This shows \eqref{claim:sec42} for each $n$.
   Moreover,  we point out   that $D_{s,y} u_{n+1}(t,x)  =0$ if $s\geq t$.
 
 To obtain the uniform estimate in \eqref{Conq1}, we  proceed with  the   finite  iterations  
  \begin{align}
   & D_{s,y} u_{n+1}(t,x)  = G_{t-s}(x-y)  \sigma\big( u_n(s,y) \big)\notag \\
   & \quad +  \int_s^t\int_{\R^2} G_{t-r_1}(x-z_1) \Sigma^{(n)}_{r_1,z_1} G_{r_1 - s}(z_1-y)  \sigma(  u_{n-1}(s,y) ) W(dr_1,dz_1) \notag \\
   &\qquad +   \sum_{k=2}^n   \int_s^t\cdots\int_s^{r_{k-1}} \int_{\R^{2k}}G_{r_k -s}(z_k-y) \sigma\big( u_{n-k}(s,y) \big) \notag \\
   &\qquad\qquad  \times \prod_{j=1}^{k} G_{r_{j-1} - r_{j}}(z_{j-1} - z_{j}) \Sigma^{(n+1-j)}_{r_j,z_j}  W(dr_j,dz_j)=:   \sum_{k=0}^n T^{(n)}_k,  \label{finiteit} 
  \end{align}   
 where $T^{(n)}_k$ denotes the $k$th item  in the sum and   $r_0=t, z_0=x$. For example,  $T_0^{(n)}=G_{t-s}(x-y)  \sigma\big( u_n(s,y) \big)$ and 
 $$
 T^{(n)}_1 = \int_s^t\int_{\R^2} G_{t-r_1}(x-z_1) \Sigma^{(n)}_{r_1,z_1} G_{r_1 - s}(z_1-y)  \sigma(  u_{n-1}(s,y) ) W(dr_1,dz_1).
 $$
We are going to estimate the $p$-norm of each of  term $T^{(n)}_k$ for $k=0, \dots, n$.
 
 \medskip
 \noindent
 {\it Case $k=0$}:  It is clear that
 \begin{equation} \label{k=0}
 \| T^{(n)}_0\|_p  \le  \kappa_{p,t} G_{t-s}(x-y),
 \end{equation}
 where $\kappa_{p,t}$ is the constant defined in \eqref{sigmap}.  
  
  \medskip
 \noindent
 {\it Case $k=1$}:  Applying 
  \eqref{eq2}, Minkowski's inequality and \eqref{Sobolev}, we can write
 \begin{align*}
&\| T^{(n)}_1\|_p^2   \leq 4p \Bigg\|  \int_s^t\int_{\R^4} G_{t-r_1}(x-z_1)G_{t-r_1}(x-z_1') G_{r_1 - s}(z_1-y) G_{r_1 - s}(z_1'-y)\\
&\qquad\qquad  \times \|z_1-z_1'\|^{-\beta}  
\Sigma^{(n)}_{r_1,z_1}  \Sigma^{(n)}_{r'_1,z'_1}  \sigma^2(  u_{n-1}(s,y) )  dz_1 dz_1' dr_1  \Bigg\|_{p/2} \\
&\le  4p L^2\kappa_{p,t}^2    \int_s^t \big\| G_{t-r_1}(x-\bullet)    G_{r_1 - s}(y-\bullet) \big\|_{\mathfrak{H}_0}^2  dr_1 \quad \text{with $\mathfrak{H}_0$ introduced in \eqref{H_0}}  \\
&\le 4pL^2\kappa_{p,t}^2  C_\beta  \int_s^t \left(\int_{\R^2} G_{t-r_1}^{2q}(x-z_1)  G_{r_1 - s}^{2q}(z_1-y)  dz_1 \right)^{1/q}dr_1, 
  \end{align*}
  with $q= 2/(4-\beta)$.  
  Then,
   we
    can
     deduce
      immediately  from Lemma \ref{LEM2} (with $\delta =1/q$)  that 
\begin{equation}   \label{k3c}
\| T^{(n)}_1\|_p^2  \le  pL^2\kappa_{p,t}^2  C^*_\beta  t^{\frac 1q-1} G_{t-s}^{2-\frac 1q}(x-y),
\end{equation}
 for some  generic constant  $C^*_\beta$, which only depends on $\beta$.
 Taking \eqref{EC1} into account, 
 we obtain
 \begin{equation}  \label{k=1}
    \| T^{(n)}_1\|_p \le   \sqrt{p} L\kappa_{p,t}  C^*_\beta  t^{\frac 1q- \frac 12} G_{t-s}(x-y). 
  \end{equation} 
 
    \medskip
    
 \noindent
 {\it Case} $k=2$:
  We can write 
  \[
T^{(n)}_2 =   \int_s^t\int_{\R^2} W(dr_1, dz_1) G_{t-r_1}(x-z_1) \Sigma^{(n)}_{r_1,z_1} N_{r_1,z_1}
  \]
  with $N_{r_1,z_1}$ defined to be 
  \begin{align*}
  N_{r_1,z_1}=
 \int_s^{r_1}  \int_{\R^{2}}G_{r_2 -s}(z_2-y) \sigma\big( u_{n-2}(r_2,z_2) \big)        G_{r_{1} - r_{2}}(z_{1} - z_{2}) \Sigma^{(n-1)}_{r_2,z_2}  W(dr_2,dz_2),
  \end{align*}
which is clearly $\mathscr{F}_{r_1}$-measurable.  
  Applying  again
  \eqref{eq2}, Minkowski's inequality and \eqref{Sobolev}, we can bound $\| T^{(n)}_2\|_p^2 $ by 
 \begin{align}  \notag
 & 4p \Bigg\|  \int_s^t\int_{\R^4} G_{t-r_1}(x-z_1)G_{t-r_1}(x-z_1')   \|z_1-z_1'\|^{-\beta}  
\Sigma^{(n)}_{r_1,z_1}  \Sigma^{(n)}_{r'_1,z'_1}   N_{r_1,z_1} N_{r_1,z'_1} dz_1 dz_1' dr_1  \Bigg\|_{p/2} \\ \notag
&\le  4p L^2   \int_s^t\int_{\R^4} G_{t-r_1}(x-z_1)G_{t-r_1}(x-z_1') G_{r_1 - s}(z_1-y) G_{r_1 - s}(z_1'-y) \\  \notag
& \qquad \times  \| N_{r_1,z_1} \|_p \| N_{r_1,z'_1}\|_p \|z_1-z_1'\|^{-\beta}   dz_1 dz_1'  dr_1 \notag  \\
&\le 4pL^2   C_\beta  \int_s^t \left(\int_{\R^2} G_{t-r_1}^{2q}(x-z_1) \| N_{r_1,z_1} \|_p^{2q}  dz_1 \right)^{1/q}dr_1.   \label{k2b}
  \end{align}
The same arguments used to obtain the bound  \eqref{k=1} for  $\|  T_1^{(n)} \| _p$ yield
\begin{equation} \label{k2a}
\| N_{r_1,z_1} \|_p  \le  \sqrt{p} L\kappa_{p,t}  C^*_\beta  t^{\frac 1q -\frac 12} G_{r_1-s}(z_1-y). 
 \end{equation}
 Substituting \eqref{k2a} into   \eqref{k2b} and applying Lemma \ref{LEM2} with $\delta =1/q$,   
  we obtain
 \begin{align*}
 \| T^{(n)}_2\|_p^2   &  \le  4pL^2   C_\beta  (\sqrt{p} L\kappa_{p,t}  C^*_\beta  t^{\frac 1q -\frac 12})^2 \int_s^t \left(\int_{\R^2} G_{t-r_1}^{2q}(x-z_1) G^{2q} _{r_1-s}(z_1-y)  dz_1 \right)^{1/q}dr_1 \\
&  \le  p^2 L^4 \kappa^2_{p,t}  C^*_\beta     t^{\frac 3q -2} G_{t-s} ^{2-\frac 1q} (x-y),
 \end{align*}
 which implies
 \begin{equation}
  \| T^{(n)}_2\|_ p \le  p L ^2\kappa_{p,t}  C^*_\beta     t^{\frac  3{2q} -1} G_{t-s} ^{1-\frac 1{2q}} (x-y).
 \end{equation}
 In view of  \eqref{EC1}, we obtain
  \begin{equation}
  \| T^{(n)}_2\|_ p \le  p L ^2\kappa_{p,t}  C^*_\beta     t^{\frac  2q -1} G_{t-s} (x-y).
 \end{equation}

   \medskip
 \noindent
 {\it Case} $3\le k\le n$:  The strategy to handle these cases will be slightly different. We need to get rid of the power  $\frac 1q$ in order to iterate the integrals in the time variables and obtain a summable series.
 We can write 
  \[
T^{(n)}_k =   \int_s^t\int_{\R^2} W(dr_1, dz_1) G_{t-r_1}(x-z_1) \Sigma^{(n)}_{r_1,z_1} \widehat{N}_{r_1,z_1}
  \]
  with $\widehat{N}_{r_1,z_1}$ defined to be 
  \begin{align*}
  \widehat{N}_{r_1,z_1}&=
 \int_{ s<r_k < \cdots <r_2<r_1} \int_{\R^{2k-2}}G_{r_k -s}(z_k-y) \sigma\big( u_{n-k}(s,y) \big)  \\
 & \qquad \times     \prod_{j=2}^{k} G_{r_{j-1} - r_{j}}(z_{j-1} - z_{j}) \Sigma^{(n+1-j)}_{r_j,z_j}  W(dr_j,dz_j),
  \end{align*}
which is   $\mathscr{F}_{r_1}$-measurable.     Then, 
we deduce from 
\eqref{eq2} and  \eqref{RHS1}  
that
  \begin{align*}
    \big\| T^{(n)}_k \big\|_p^2 &  \le 4p   \Bigg\|  \int_s^t dr_1      \int_{\R^4} G_{t-r_1}(x-z_1) \Sigma^{(n)}_{r_1,z_1} \widehat{N}_{r_1,z_1}G_{t-r_1}(x-z'_1) \Sigma^{(n)}_{r_1,z'_1} \widehat{N}_{r_1,z'_1}   \\
     &\qquad\qquad \times  \| z_1' - z_1\|^{-\beta}  dz_1dz_1'     \Bigg\|_{p/2} \\
     & \le  4p K_\beta L^2  t^{\frac{(2-2q)^2}{2q}}    \int_s^t dr_1   \int_{\R^2}dz_1 G^{2q}_{t-r_1}(x-z_1) \|   \widehat{N}_{r_1,z_1} \|_p^2.
  \end{align*}
   Now we can iterate the above process to obtain
      \begin{align}  \notag
 \big\| T^{(n)}_k  \big\|_p^2  &\leq   \left( 4p L^2  K_\beta t^{\frac{(2-2q)^2}{2q}} \right)^{k-1}  \int_s^tdr_1 \int_s^{r_1}dr_2\cdots \int_s^{r_{k-2}} dr_{k-1}  \int_{\R^{2k-2}}dz_1\cdots dz_{k-1}   \\  
 &  \times      G^{2q}_{t-r_1}(x-z_1)  G^{2q}_{r_1-r_2}(z_1-z_2) \cdots   G^{2q}_{r_{k-2}-r_{k-1}}(z_{k-2}-z_{k-1})     \big\|   \widetilde{N}_{r_{k-1},z_{k-1}}      \big\|_{p}^2,\label{eq8}
  \end{align}
  where    $ \widetilde{N}_{r_{k-1},z_{k-1}} $ is defined to be   
  \begin{align*}
  \int_s ^{r_{k-1} } \int_{\R^2}  W(dr_k, dz_k) \sigma\big( u_{n-k}(s,y)  \big)G_{r_{k-1}-r_k}(z_{k-1}-z_k) \Sigma^{(n+1-k)}_{r_k,z_k} G_{r_k-s}(z_k-y).
   \end{align*}
Therefore, the same arguments for estimating $\| T_1^{(n)}\|_p^2$  (see  \eqref{k3c}),  lead to 
\begin{equation} \label{eq7}
 \big\|   \widetilde{N}_{r_{k-1},z_{k-1}}      \big\|_{p}^2 \le  
 p\kappa_{p,t}^2  L^2 C^*_\beta  t^{\frac  1q-1}  G^{2-\frac 1q}_{r_{k-1} -s}\big(z _{k-1} - y),
 \end{equation}
 with  $C^*_\beta$ being a generic constant  that only depends on $\beta$.
 On the other hand, applying Lemma  \ref{LEM2} with $\delta =1$, we can write
 \begin{align}  \notag
& \int_{r_{k-1}}^{r_{k-3}} dr_{k-2}  \int_{\R^2}  dz_{k-2}G^{2q}_{r_{k-3} - r_{k-2}} (z_{k-3} - z_{k-2})  G^{2q}_{r_{k-2} - r_{k-1}} (z_{k-2} - z_{k-1}) \\  
&\qquad  \lesssim  t^{2-2q} G^{2q-1}_{r_{k-3} -r_{k-1}}( z_{k-3} -z_{k-1})  \label{k3d},
 \end{align}
 with the convention $z_0=x$ and $r_0=t$.   Plugging the estimates \eqref{eq7} and \eqref{k3d} into  \eqref{eq8}, yields
    \begin{align*}  
 \big\| T^{(n)}_k  \big\|_p^2  &\leq   \kappa_{p,t}^2  \left( p L^2  C^*_\beta t^{\frac {2(1-q)^2}q} \right)^k t^{5-4q-\frac1q} \\
 & \quad \times  \int_s^tdr_1 \int_s^{r_1}dr_2\cdots \int_s^{r_{k-3}} dr_{k-1}  \int_{\R^{2k-4}}dz_1\cdots  dz_{k-3} dz_{k-1}     \\
 & \quad \times   G^{2q}_{t-r_1}(x-z_1)   \cdots   G^{2q}_{r_{k-4}-r_{k-3}}(z_{k-4}-z_{k-3})  \\
 & \quad \times  G^{2q-1}_{r_{k-3}-r_{k-1}}(z_{k-3}-z_{k-1})      
  G^{2-\frac 1q}_{r_{k-1} -s}\big(z _{k-1} - y)
  \end{align*}
 By Cauchy-Schwartz inequality and \eqref{fact01},
 \begin{align*}
 & \int_{\R^2} G^{2q-1}_{r_{k-3}-r_{k-1}}(z_{k-3}-z_{k-1})      
  G^{2-\frac 1q}_{r_{k-1} -s}\big(z _{k-1} - y) dz_{k-1} \\
  & \qquad 
  \le  \left[ \int_{\R^2} G^{4q-2}_{r_{k-3}-r_{k-1}}(z)   dz
  \int_{\R^2}G^{4-\frac 2q}_{r_{k-1} -s}(z)  dz \right]^{1/2}  \le  C^*_\beta  t^{1-2q+\frac1q}.
  \end{align*}
In this way, we obtain
  \begin{align}    
 \big\| T^{(n)}_k  \big\|_p^2  &\leq   \kappa_{p,t}^2  ( p L^2  C^*_\beta t^{\frac {2(1-q)^2}q} )^k t^{6(1-q)}  
 {\bf 1}_{\{ \|x -y\| < t -s   \}}    \label{ind-why}  \\   \notag
&  \qquad\times  \int_s^tdr_1 \int_s^{r_1}dr_2\cdots \int_s^{r_{k-3}}  dr_{k-1}  \int_{\R^{2k-6}}dz_1\cdots  dz_{k-3}      \\
 & \qquad \times   G^{2q}_{t-r_1}(x-z_1)   \cdots   G^{2q}_{r_{k-4}-r_{k-3}}(z_{k-4}-z_{k-3})    \notag
  \end{align}
 The indicator function $ {\bf 1}_{\{ \|x -y\| < t -s   \}}$ appears in \eqref{ind-why}, because   
 \begin{center}
$   \mathbf{1}_{ \{  \|z_{k-1} -y\| < r_{k-1}-s, \|z_{k-3} -z_{k-1}\| < r_{k-3}- r_{k-1}, \ldots, \|x-z_1\|< t-r_1   \}  }  \leq  {\bf 1}_{\{ \|x -y\| < t -s   \}}$.
 \end{center}

 Now,  we can perform the integration with respect to $dz_{k-3}, \dots, dz_1$ one by one to get
 \begin{align*}
 & \int_{\R^{2k-6}}   dz_1\cdots dz_{k-3}     G^{2q}_{t-r_1}(x-z_1)  G^{2q}_{r_1-r_2}(z_1-z_2) \cdots   G^{2q}_{r_{k-4}-r_{k-3}}(z_{k-4}-z_{k-3}) \\
 &\qquad =\left( \frac{(2\pi)^{1-2q}  }{2-2q} \right)^{k-3} \times \prod_{j=1}^{k-3} \big( r_{j-1}  - r_j\big)^{2-2q} \leq \left( \frac{(2\pi)^{1-2q}  }{2-2q} t^{2-2q} \right)^{k-3},
 \end{align*}
 in view of the equality \eqref{fact01}.     Together with the integration on the simplex $\{ t> r_1 >  \cdots > r_{k-3} >r_{k-1}>s\}$, we get
 \begin{equation*}  
  \big\| T^{(n)}_k  \big\|_p^2 \leq  \frac{ (pC_\beta^* L^2)^{k}    }{(k-2)!}  \kappa_{p,t}^2
   t^{  k(  \frac 2q -1)-2}  {\bf 1}_{\{ \|x -y\| < t -s   \}}.
  \end{equation*}
 Thus, taking into account that  
  $$
   {\bf 1}_{\{ \|x -y\| < t -s   \}} \leq \big[2\pi (t-s)\big]^2 G^2_{t-s}(x-y), 
   $$
  we obtain for $k\in\{3,\dots,n\}$,
  \begin{align}\label{kle3}
   \big\| T^{(n)}_k \big\|_p &\leq  \kappa_{p,t} \frac{ (p C^\ast_\beta L^2)^{k/2}  }{ \sqrt{(k-2)!}}  t^{\ k(\frac 1q -\frac 12)} G_{t-s}(x-y). 
  \end{align}
  Hence, we deduce from \eqref{k=0}, \eqref{k=1} and \eqref{kle3} that for any $n\geq 3$, 
\[
  \big\| D_{s,y} u_{n+1}(t,x) \big\|_p \leq \sum_{k=0}^n \big\| T^{(n)}_k  \big\|_p  
\le C_{\beta,p, t, L}  \kappa_{p,t} G_{t-s} (x-y),
\]
where the constant  $C_{\beta,p, t, L}$  is defined in \eqref{Cbeta}. This proves Proposition  \ref{PROP32}. \qedhere
 \end{proof}

 \medskip
 
 \noindent
 {\bf Step 2.}  We are going to show that $D_{s,y} u(t,x)$ is a real-valued random variable.
 As a consequence of \eqref{Sobolev}, \eqref{Conq1} and  \eqref{fact01}, we have for any $p\geq 2$ and with $q=2 /(4-\beta)$
 \begin{align*}
 & \E\Big[ \| Du_{n+1}(t,x) \| _{\mathfrak{H}}^p \Big] ^{2/p}= \left\| \int_{\R_+} ds \big\| D_{s,\bullet} u_{n+1}(t,x)\big\|^2_{\mathfrak{H}_0} \right\|_{p/2}\\
 & \lesssim  \left\| \int_{\R_+} ds \left( \int_{\R^2} \vert D_{s,y} u_{n+1}(t,x)\vert^{2q} dy \right)^{1/q} \right\|_{p/2} \\
 &\lesssim \int_{\R_+} ds \left( \int_{\R^2} \big\| D_{s,y} u_{n+1}(t,x)\big\|^{2q}_p dy \right)^{1/q}\text{\small by applying Minkowski twice}\\
 &\lesssim \int_{\R_+} ds \left( \int_{\R^2} G_{t-s}^{2q}(x-y) dy \right)^{1/q} \lesssim  \int_{0}^t (t-s)^{\frac{2-2q}{q} }ds \lesssim 1.
 \end{align*}
 One can first read from the above estimates that  $\{Du_{n+1}(t,x),  n\geq 1\} $ is uniformly bounded in   $L^p\big(\Omega; \mathfrak{H}   \big)$, which together with the $L^p$-convergence of $u_n(t,x)$ to $u(t,x)$ implies the   convergence of $Du_{n+1}(t,x)$ to $Du(t,x)$ in the weak topology on $L^p\big(\Omega; \mathfrak{H}   \big)$ up to a subsequence; this fact is well-known in the literature, see for instance \cite{MilletMarta}. One can deduce from the same arguments that $\{ Du_{n+1}(t,x), n\geq 1\} $ is uniformly bounded in   $L^p\big(\Omega; L^{2q}(\R_+\times\R^2)   \big)$:
 \begin{align*}
&\big\|  Du_{n+1}(t,x) \big\|_{L^p \left(\Omega; L^{2q}(\R_+\times\R^2)   \right)}^p  =    \left\| \int_{\R_+\times\R^2} \vert D_{s,y} u_{n+1}(t,x)\vert^{2q} dy ds  \right\|_{\frac{p}{2q}}^{\frac{p}{2q}}   \\
&\qquad \leq  \left( \int_{\R_+\times\R^2} \big\| D_{s,y} u_{n+1}(t,x)\big\|^{2q}_p dy ds \right)^{\frac{p}{2q}} \lesssim  \left( \int_{\R_+\times\R^2}G^{2q}_{t-s}(x-y) dy ds \right)^{\frac{p}{2q}} \lesssim 1.
 \end{align*} 
So up to a subsequence, $Du_{n}(t,x)$ also converges to $Du(t,x)$ in the weak topology on $L^p\big(\Omega;  L^{2q}(\R_+\times\R^2)   \big)$.  In particular, 
 we have ($2q< 2 \leq p <\infty$)
 \begin{align}\label{final1}
 \sup_{(t,x)\in[0,T]\times\R^2}  \left\|   \int_{\R_+\times\R^2} \vert D_{s,y} u(t,x)\vert^{2q} dy ds    \right\|_{\frac{p}{2q}} < +\infty
 \end{align}
 and $D_{s,y}u(t,x)$ is a real function in $(s,y)$.

 \medskip
 \noindent
 {\bf Step 3.}  Let us prove the lower bound. 
 By Lemma \ref{CO}, we can write
\[
u(t,x) -1 = \int_0^t\int_{\R^2} \E\big[ D_{s,y}u(t,x) \vert \F_s \big] W(ds,dy),
\]
so that a comparison with \eqref{mild} yields $\E\big[ D_{s,y}u(t,x) \vert \F_s \big] = G_{t-s}(x-y) \sigma(u(s,y))$ almost everywhere in $\Omega \times \R_+ \times \R^2$.  It follows that
\[
\big\| \E [ D_{s,y}u(t,x) \vert \F_s  ] \big\|_p = G_{t-s}(x-y) \big\|\sigma(u_{s,y})\big\|_p, 
\]
thus by conditional Jensen, we have
\[
\big\|   D_{s,y}u(t,x)   \big\|_p \geq G_{t-s}(x-y) \big\|\sigma(u_{s,y})\big\|_p,
\]
which is exactly the lower bound in   \eqref{IMP}.

 \medskip
 \noindent
  {\bf Step 4.} 
We are finally in a position to prove the upper bound in  \eqref{IMP}.  
  Put $p^\star = p/(p-1)$, which is the conjugate exponent for $p$.  
Let us pick a nonnegative function $M\in C_c(\R_+\times\R^2)$  and 
random variable  $\mathcal{Z}\in L^{p^\star}(\Omega)$ with $\| \mathcal{Z}\|_{p^\star} \leq 1$.
Since $Du_{n}(t,x)$   converges to $Du(t,x)$ in the weak topology on $L^p\big(\Omega;  L^{2q}(\R_+\times\R^2)   \big)$ along some subsequence (say $Du_{n_k}(t,x)$), we have, in view of \eqref{Conq1}
\begin{align*}
 & \int_{\R_+\times \R^2} M(s,y) \E\big[  Z D_{s,y} u(t,x) \big] dsdy =
\lim_{k\rightarrow \infty} \int_{\R_+\times \R^2} M(s,y) \E\big[  Z D_{s,y} u_{n_k}(t,x) \big] dsdy \\
& \qquad \le C_{\beta,p, t, L}  \kappa_{p,t} \int_{\R_+\times \R^2} M(s,y) G_{t-s}(x-y) ds dy.
\end{align*}
This implies that for almost all $(s,y)\in [0,t\times \R^2$,
\[
 \E\big[  Z D_{s,y} u(t,x) \big] \le  C_{\beta,p, t, L}  \kappa_{p,t} G_{t-s}(x-y)
\]
Taking the supremum over $\{ \mathcal{Z}: \| \mathcal{Z}\|_{p^\star} \leq 1  \}$  yields
\[
 \| D_{s,y} u(t,x) \|_p  \le  C_{\beta,p, t, L}  \kappa_{p,t} G_{t-s}(x-y),
\]
which finishes the proof.
 \qedhere

\subsection{Proof of technical lemmas} \label{lemmas}  For   convenience,  let us recall Lemma \ref{LEM1} below.

 \quad\\
\noindent{\bf Lemma \ref{LEM1}}.  For  $t>s$, with $\| z\|= \w >0$ and $q\in(1/2,1)$
\begin{align*}
\qquad G_t^{2q}\ast G_s^{2q}(z) & \lesssim   {\bf 1}_{\{ \w < s\}}  \big[ t^2 - (s-\w)^2\big]^{1-2q} +       \big[ t^2 - (s+\w)^2 \big]^{1-2q } {\bf 1}_{\{ t> s+\w  \}} \notag \\
 &\quad +  {\bf 1}_{\{\vert s-\w \vert < t < s+\w  \}} \big[ (\w+s)^2 -t^2\big]^{-q+\frac12} \big[t^2- (s-\w)^2\big]^{-q+\frac12},\quad\qquad \eqref{qq}  
 \end{align*}
  where the implicit constant depends only on $q$.

\begin{proof}[Proof of Lemma \ref{LEM1}]   We are interested in estimating 
\[
\mathbf{I}= \int_{\R^2} \big(t^2- \|x\|^2\big)_+^{-q} \big(s^2-\|x-z\|^2\big)_+^{-q} dx,
\]
where  $(v)_+^{-q}= v^{-q}$ for $v>0$ and $(v)_+^{-q}=0$ for $v\leq 0$.   Because the convolution of two radial functions is radial, the quantity $\mathbf{I}$ depends only on $s$, $t$ and $\|z\|$. Hence,  we can assume additionally  that $z=(\w,0)$, where $\w>0$.  Note that the integral  $\mathbf{I}$  vanishes if $t+s {\color{blue}<} \w$ and we can write, putting $x=(\xi, \eta)$,
\[
\mathbf{I}=  \int_{\R^2} \big(t^2- \xi^2-\eta^2\big)_+^{-q} \big(s^2-(\xi-\w)^2-\eta^2\big)_+^{-q} d\xi d\eta.
\]
Making the change of variables
$
(x,y)=  \big( \xi^2+ \eta^2,  (\w-\xi)^2+ \eta^2\big)
$
yields
\begin{equation} \label{I}
\mathbf{I}=\frac{1}{2} \int_D  (t^2-x)^{-q} (s^2-y)^{-q}  \big[ ( \sqrt{x} +\w)^2 -y\big]^{-  1/2} \big[ y-( \sqrt{x}-\w)^2\big]^{-  1/2} dxdy,
\end{equation}
where 
\[
D =\left\{ (x,y)\in\R^2: 0< x <t^2, 0 < y < s^2, \big(\sqrt{x}-\w\big)^2 <y< \big(\sqrt{x}+\w\big)^2\right\}.
\]
To derive the expression \eqref{I}  for $\mathbf{I}$, we have used the fact that the Jacobian of the change of variables is
\[
\left| \frac { \partial (x,y)}{\partial (\xi, \eta)} \right| =
4 \w  |\eta| = 2 \big[ ( \sqrt{x} +\w)^2 -y\big]^{ 1/2} \big[ y-( \sqrt{x}-\w)^2\big]^{  1/2}.
\]
Then, integrating first in the variable $y$ yields
\begin{align}
\mathbf{I}&=  \frac{1}{2} \int_0^{t^2} dx (t^2-x)^{-q} \int_{D(x)} dy ~   (s^2-y)^{-q}  \big[ ( \sqrt{x} +\w)^2 -y\big]^{-  1/2} \big[ y-( \sqrt{x}-\w)^2\big]^{-  1/2}\notag \\
&=:\frac{1}{2}\int_0^{t^2} (t^2-x)^{-q} \mathcal{S}_q(x)dx,\notag  
\end{align}
where 
 \[
D(x)=\big\{ y\in\R: (x,y)\in D\big\} = \Big\{ y\in\R: y <s^2 , \big(\sqrt{x}-\w\big)^2 <y< \big(\sqrt{x}+\w\big)^2\Big\}
\]
and
\begin{equation}
\mathcal{S}_q(x) =\int_{D(x)} dy ~   (s^2-y)^{-q}  \big[ ( \sqrt{x} +\w)^2 -y\big]^{-  1/2} \big[ y-( \sqrt{x}-\w)^2\big]^{-  1/2}.\label{Sq} 
\end{equation}

Let us first deal with $\mathcal{S}_q(x)$ for every $x\in(0, t^2)$. There are two possible cases, depending on the value of $x$:

\medskip
\noindent
{\bf   (A)} When $(\sqrt{x} -\w)^2 <s^2 < (\sqrt{x}+\w)^2$,  
\begin{align}  \notag
\mathcal{S}_q(x)=& \int_{(\sqrt{x} -\w)^2} ^{s^2} (s^2-y)^{-q}  \big[( \sqrt{x} +\w)^2 -y\big]^{-  1/2} \big[y-( \sqrt{x}-\w)^2\big]^{-  1/2} dy \\  \notag
&\leq {\rm Beta}(1/2, 1-q)  \big[ ( \sqrt{x} +\w)^2 -s^2)\big]^{-  1/2} \big[ s^2-( \sqrt{x}-\w)^2\big]^{-q+\frac 12}\\  \label{A}
&\lesssim  \big[ ( \sqrt{x} +\w)^2 -s^2)\big]^{-  1/2} \big[ s^2-( \sqrt{x}-\w)^2\big]^{-q+\frac 12}.
\end{align}
 Throughout this section,  ${\rm Beta}(a,b )$ denotes the usual beta function:
\[
{\rm Beta}(a,b ) = \int_0^1 x^{a-1} (1-x)^{b-1} dx, ~ a,b\in(0,\infty).
\]

\medskip
\noindent
{\bf   (B)} When $(\sqrt{x} -\w)^2 < (\sqrt{x}+\w)^2<s^2$,
\begin{align*}
\mathcal{S}_q(x)&= \int_{(\sqrt{x} -\w)^2} ^{(\sqrt{x}+\w)^2}  (s^2-y)^{-q}  \big[( \sqrt{x} +\w)^2 -y\big]^{-  1/2} \big[y-( \sqrt{x}-\w)^2\big]^{-  1/2} dy   \\
&\leq (s^2-( \sqrt{x} +\w)^2 )^{-q}   \int_{(\sqrt{x} -\w)^2} ^{(\sqrt{x}+\w)^2} \big[( \sqrt{x} +\w)^2 -y\big]^{-  1/2} \big[y-( \sqrt{x}-\w)^2\big]^{-  1/2} dy\\
& = {\rm Beta}(1/2, 1/2)   \big[s^2-( \sqrt{x} +\w)^2 \big]^{-q}   \lesssim  \big[s^2-( \sqrt{x} +\w)^2 \big]^{-q}.
\end{align*}

\medskip
Note that three positive numbers $a,b,c$ can form sides of a triangle if and only if {\it the sum of any two of them is strictly bigger than the third one}, which is equivalent to saying that $\vert a-b\vert < c <a+b$. It follows that
\begin{align*}
(\sqrt{x} -\w)^2 <s^2 < (\sqrt{x}+\w)^2 &\Leftrightarrow  \text{$\sqrt{x}, \w, s$ can be the sides of a triangle} \\
&\Leftrightarrow   (s-\w)^2< x < (s+\w)^2.
\end{align*}
Furthermore, it is trivial that $(\sqrt{x} -\w)^2 < (\sqrt{x}+\w)^2<s^2 \Leftrightarrow  x< (s-\w)^2~ \text{and} ~s>\w$.

Now we decompose  the integral $2\mathbf{I} = \int_0^{t^2} (t^2-x)^{-q} \mathcal{S}_q(x)dx$ into two parts corresponding to    the cases {\bf   (A)}  and {\bf   (B)}:
\[
2\mathbf{I}  = \mathbf{I}_{\bf A} + \mathbf{I}_{\bf B},
\]
where 
\begin{align*}
\mathbf{I}_{\bf A}= \int_{(s-\w)^2}^{t^2\wedge  (s+\w)^2 }  (t^2-x)^{-q} \mathcal{S}_q(x)dx
\quad {\rm and} \quad 
\mathbf{I}_{\bf B}= \int_0^{(s-\w)^2\wedge t^2} (t^2-x)^{-q} \mathcal{S}_q(x)dx.
\end{align*}

\medskip
\noindent
{\bf Estimation of $\mathbf{I}_{\bf A}$.}   
We first write, using \eqref{A}, 
\begin{align*}
\mathbf{I}_{\bf A}& \lesssim \int_{(s-\w)^2}^{t^2\wedge  (s+\w)^2 } (t^2-x)^{-q}  \big[ ( \sqrt{x} +\w)^2 -s^2)\big]^{-  1/2} \big[ s^2-( \sqrt{x}-\w)^2\big]^{-q+\frac 12} dx\\
& = \int_{(s-\w)^2}^{t^2\wedge  (s+\w)^2 } (t^2-x)^{-q}        \big[ (\w+s)^2-x\big]^{-q+\frac 12}  \big[ x- (\w-s)^2\big]^{-q+\frac 12}  \big[ (\sqrt{x} +\w)^2-s^2\big]^{q-1}dx.
\end{align*}
Recall in this case $\sqrt{x} +\w> s$,  which implies $(\sqrt{x} +\w)^2-s^2 > x- (s-\w)^2>0$. Therefore,
\begin{align*}  
\mathbf{I}_{\bf A}& \lesssim  \int_{(s-\w)^2}^{t^2\wedge  (s+\w)^2 } (t^2-x)^{-q}        \big[ (\w+s)^2-x\big]^{-q+\frac 12}  \big[ x- (\w-s)^2\big]^{-1/2}  dx.
\end{align*}
Now we consider the following two sub-cases:
 \begin{itemize}
\item[\bf (A1)] If   $s+\w<t$, then  for $ (s-\w)^2 < x < (s+\w)^2 <t$, we have, with $\gamma = 2- q^{-1}$, 
 \begin{align*}
(t^2-x)^{-q}     &\leq   \big[ t^2 - (s+\w)^2 \big]^{-q\gamma} \big[ (s+\w)^2 -x \big]^{-q+q\gamma} \\
&= \big[ t^2 - (s+\w)^2 \big]^{1-2q} \big[ (s+\w)^2 -x \big]^{q-1}.
\end{align*}
 Thus, 
 \begin{align*}
\mathbf{I}_{\bf A}& \lesssim  \big[ t^2 - (s+\w)^2 \big]^{1-2q }  \int_{(s-\w)^2}^{ (s+\w)^2 }       \big[ (\w+s)^2-x\big]^{-1/2}  \big[ x- (\w-s)^2\big]^{-1/2}  dx \\
&=  {\rm Beta}(1/2, 1/2)  \big[ t^2 - (s+\w)^2 \big]^{1-2q }.
\end{align*}
  
  \medskip
  
\item[\bf (A2)] If  $(s-\w)^2 <t^2< (s+\w)^2$ (\emph{i.e.} $s,\w,t$ form triangle sides), then  
 \begin{align*}
\mathbf{I}_{\bf A} \lesssim  & \int_{(s-\w)^2} ^{t^2}  (t^2-x)^{-q}  \big[ (\w+s)^2-x\big]^{-q+\frac 12}  \big[x- (\w-s)^2\big]^{-  1/2}dx \\
 \leq &   \big[ (\w+s)^2 -t^2\big]^{-q+\frac12}  \int_{(s-\w)^2} ^{t^2}       \big(  t^2 -x\big)^ {-q }   \big[x- (\w-s)^2\big]^{-  1/2}dx \\
 \lesssim&   \big[ (\w+s)^2 -t^2\big]^{-q+\frac12} \big[t^2- (s-\w)^2\big]^{-q+\frac12}
\end{align*}
because 
$
  \int_{a} ^{b}       (  b -x )^ {-q }    (x- a)^{-  1/2}dx = \text{Beta}(1/2, 1-q) (b-a)^{-q+\frac{1}{2}}
$ 
for any $0 \leq a<b<\infty$ and for any $q<1$.
\end{itemize}
Combining {\bf (A1)} and  {\bf (A2)}, we have obtained
 \begin{equation} \label{I_A}
 \mathbf{I}_{\bf A} \lesssim \big[ t^2 - (s+\w)^2 \big]^{1-2q } {\bf 1}_{\{ t> s+\w  \}} +  {\bf 1}_{\{\vert s-\w \vert < t < s+\w  \}} \big[ (\w+s)^2 -t^2\big]^{\frac{1-2q}{2}} \big[t^2- (s-\w)^2\big]^{\frac{1-2q}{2}}.
 \end{equation}

\bigskip

\noindent{\bf Estimation of $\mathbf{I}_{\bf B}$.}   In this case, $\sqrt{x}<s-\w$ and $\w <s$, then  
\[
s^2- (\sqrt{x }+\w)^2 > (s-\w)^2-x > 0.
\]
Therefore, $\mathcal{S}_q(x)\lesssim   \big[(s-\w)^2-x \big]^{-q}  $ and the quantity $\mathbf{I}_{\bf B}$ can be bounded as follows
\begin{align} \notag
 {\bf I_B} & =  \int_0^{(s-\w)^2} (t^2-x)^{-q} \mathcal{S}_q(x)dx   \lesssim  \int_0^{(s-\w)^2} (t^2-x)^{-q}     \big[(s-\w)^2-x \big]^{-q}dx\\
 &   \lesssim \big[ t^2 - (s-\w)^2\big]^{1-2q},    \label{I_B}
\end{align}
 because for any $0<a<b <\infty$ and any $p,q\in(1/2, 1)$
 \begin{align} 
 \int_0^a (b-x)^{-p} (a-x)^{-q}dx & = \int_0^a (b-a+y)^{-p} y^{-q}dy = (b-a)^{1-p-q}\int_0^{\frac{a}{b-a}} y^{-q}(1+y)^{-p}dy  \notag \\
 &\leq (b-a)^{1-p-q}\int_0^{\infty} y^{-q}(1+y)^{-p}dy  \lesssim (b-a)^{1-p-q}. \notag   
 \end{align}
Our proof is done by combining the  estimates  \eqref{I_A} and \eqref{I_B} to get \eqref{qq}.   \qedhere

 \end{proof}

 Now let us apply Lemma \ref{LEM1} to prove Lemma \ref{LEM2}.

 \begin{proof}[Proof of  Lemma \ref{LEM2}]  Put $\mu = (t-r)\wedge (r-s)$ and $\nu=(t-r)\vee (r-s)$ and assume $\mu \neq \nu$. We apply Lemma \ref{LEM1}  to write 
 \begin{align*}
\big( G_{t-r}^{2q} \ast G_{r-s}^{2q}(z) \big)^{\delta}& \lesssim \Big({\bf 1}_{\{ \w < \mu\}}  \big[ \nu^2 - (\mu-\w)^2\big]^{1-2q} +       \big[ \nu^2 - (\mu+\w)^2 \big]^{1-2q } {\bf 1}_{\{ \nu> \mu+\w  \}} \notag \\
 &\,\,\, +  {\bf 1}_{\{\vert \mu-\w \vert < \nu < \mu+\w  \}} \big[ (\w+\mu)^2 -\nu^2\big]^{-q+\frac12} \big[\nu^2- (\mu-\w)^2\big]^{-q+\frac12}\Big)^{\delta} \\
 & \lesssim  {\bf 1}_{\{ \w < \mu\}}  \big[ \nu^2 - (\mu-\w)^2\big]^{\delta(1-2q) } +       \big[ \nu^2 - (\mu+\w)^2 \big]^{\delta(1-2q)  } {\bf 1}_{\{ \nu> \mu+\w  \}} \notag \\
 &\,\,\, +  {\bf 1}_{\{\vert \mu-\w \vert < \nu < \mu+\w  \}} \big[ (\w+\mu)^2 -\nu^2\big]^{\delta(\frac 12-q)} \big[\nu^2- (\mu-\w)^2\big]^{\delta(\frac 12-q) },
 \end{align*}
 where $\w = \|z\| >0$ and $0> \delta(1-2q) \ge \frac 1q-2  >-1$.     Define 
 \begin{align*}
  K^{(1)}_{s,t}(z):&= \int_s^t dr {\bf 1}_{\{ \w < \mu\}}  \big[ \nu^2 - (\mu-\w)^2\big]^{\delta(1-2q) } \\
  & = \int_s^t dr {\bf 1}_{\{ \w < \mu\}}  \big[ (\nu + \mu-\w)(\nu-\mu+\w)\big]^{\delta(1-2q) }
  \end{align*}
 and note that $t-r > r-s$ if and only if $r < \frac{t+s}{2}$. Then, by exact computations and decomposing the integral in the
 intervals $[s, (t+s)/2] $ and $ [(t+s)/2,t]$, yields
  \begin{align}   \notag
  K^{(1)}_{s,t}(z)&= {\bf 1}_{\{ \w < \frac{t-s}{2}\}} \int_{s+\w}^{ (t+s)/2 } (t-s-\w)^{\delta(1-2q) }  (t+s+\w - 2r)^{\delta(1-2q) } dr \\ \notag
  &\qquad  + {\bf 1}_{\{ \w < \frac{t-s}{2}\}}  \int_{ (t+s)/2 }^{t-\w}  (t-s-\w)^{\delta(1-2q) }  ( 2r + \w -  t-  s)^{\delta(1-2q)} dr \\ \notag
  &= 2\times  {\bf 1}_{\{ \w < \frac{t-s}{2}\}}  (t-s-\w)^{\delta(1-2q) } \frac{1}{2(\delta(1-2q)+1)} \\  \notag
  & \qquad \times \left[    (t-s-\w)^{\delta(1-2q)+1} - \w^{\delta(1-2q)+1} \right] \\ \notag
  &  \le   \frac{ (t-s)^{\delta(1-2q)+1}} { \delta(1-2q)+1}  {\bf 1}_{\{ \w < \frac{t-s}{2}\}}  (t-s-\w)^{\delta(1-2q) } \\ \notag
&  \lesssim  (t-s)^{\delta(1-2q)+1}(t-s)^{\delta(1-2q)}   {\bf 1}_{\{ \w < \frac{t-s}{2}\}} \\ \label{K1}
&  \lesssim  (t-s)^{\delta(1-2q)+1}\big[ (t-s)^2- \|z\|^2\big]^{\delta(\frac 12-q)} {\bf 1}_{\{ \|z\| <t-s  \}}.
     \end{align}
 By the same arguments, we can get
 \begin{align}    \notag
   K^{(2)}_{s,t}(z):&=  \int_s^tdr   \big[ \nu^2 - (\mu+\w)^2 \big]^{\delta(1-2q) } {\bf 1}_{\{ \nu> \mu+\w  \}}  \\ \notag
   &= \int_s^tdr   \big[ (\nu + \mu+\w)(\nu-\mu-\w) \big]^{\delta(1-2q) } {\bf 1}_{\{ \nu> \mu+\w  \}}  \\ \notag
   &=  {\bf 1}_{\{ t-s> \w  \}} (t-s+\w)^ {\delta(1-2q) } \int_s^{(t+s-\w)/2} \big( t+s-2r-\w\big)^{\delta(1-2q) } dr \\ \notag
   &\qquad +  {\bf 1}_{\{ t-s> \w  \}} (t-s+\w)^ {\delta(1-2q) }  \int_{(t+s+\w)/2}^t \big( 2r-s-t-\w\big)^{\delta(1-2q) } dr \\ \notag
   &= {\bf 1}_{\{ t-s> \w  \}} (t-s+\w)^{\delta(1-2q) }     \frac{1}{2(\delta(1-2q)+1)} (t-s-\w)^{\delta(1-2q)+1} \times 2 \\ \label{K2}
   &\lesssim    (t-s)^{\delta(1-2q)+1} \big[ (t-s)^2- \|z\|^2\big]^{\delta(\frac 12-q)} {\bf 1}_{\{ \|z\| <t-s  \}}. 
 \end{align}
 Similarly,  we first write 
  \begin{align*}
   K^{(3)}_{s,t}(z):&=\int_s^tdr   {\bf 1}_{\{\vert \mu-\w \vert < \nu < \mu+\w  \}} \big[ (\w+\mu)^2 -\nu^2\big]^{\delta(\frac 12-q)} \big[\nu^2- (\mu-\w)^2\big]^{\delta(\frac 12-q)} \\
   &= \int_s^tdr   {\bf 1}_{\{   \nu-\mu <\w < \mu+\nu      \}} \big[ (\mu+\nu)^2 - \w^2 \big]^{\delta(\frac 12-q)} (\w+\mu-\nu)^{\delta(\frac 12-q)} (\w+\nu-\mu)^{\delta(\frac 12-q)} \\
   &=\big[ (t-s)^2 - \w^2 \big]^{\delta(\frac 12-q)}   \int_s^tdr   {\bf 1}_{\{   \nu-\mu <\w < \mu+\nu      \}} (\w+\mu-\nu)^{\delta(\frac 12-q)} (\w+\nu-\mu)^{\delta(\frac 12-q)}.
   \end{align*}
Recall $t-r > r-s$ if and only if $r < \frac{t+s}{2}$. Then
 \begin{align*}
 &\quad  \int_s^{(t+s)/2} dr   {\bf 1}_{\{   \nu-\mu <\w < \mu+\nu      \}} (\w+\mu-\nu)^{\delta(\frac 12-q)} (\w+\nu-\mu)^{\delta(\frac 12-q)} \\
 &= {\bf 1}_{\{   \w < t-s     \}}  \int_{\frac{t+s-\w}{2}  }^{\frac{t+s}{2} } dr \, (\w - t - s + 2r)^{\delta(\frac 12-q)} (\w+t+s-2r)^{\delta(\frac 12-q)} \\
 &= {\bf 1}_{\{   \w < t-s     \}}  2^{\delta(1-2q)} \int_a^b (r-a)^{-\delta(\frac 12-q)} (c-r)^{\delta(\frac 12-q)}dr,
  \end{align*}
 where $a= \dfrac{t+s-\w}{2}<b=\dfrac{t+s}{2} < c=\dfrac{t+s+\w}{2}$. It is easy to show that 
 \begin{align*}
  \int_a^b (r-a)^{\delta(\frac 12-q)} (c-r)^{\delta(\frac 12-q)}dr&  = (c-a)^{\delta(1-2q)+1} \int_0^{\frac{b-a}{c-a}} t^{\delta(\frac 12-q)} (1-t)^{\delta(\frac 12-q)} dt\\
  & \leq (c-a)^{\delta(1-2q)+1} \int_0^{1} t^{\delta(\frac 12-q)} (1-t)^{\delta(\frac 12-q)} dt\\
  &=  \text{Beta}(\delta(\frac 12-q)+1, \delta(\frac 12-q)+1))  (c-a)^{\delta(1-2q)+1}.
  \end{align*}
 Therefore,
  \begin{align*}
 & \int_s^{(t+s)/2} dr   {\bf 1}_{\{   \nu-\mu <\w < \mu+\nu      \}} (\w+\mu-\nu)^{\delta(\frac 12-q)} (\w+\nu-\mu)^{\delta(\frac 12-q)}  \\ 
 &\qquad  \lesssim {\bf 1}_{\{   \w < t-s     \}}  \w^{\delta(1-2q)+1} \leq (t-s)^{\delta(1-2q)+1} {\bf 1}_{\{ \|z\| <t-s  \}}.
   \end{align*}
 In the same manner, we can get
  \begin{align}
 &  \int_{(t+s)/2}^{t} dr   {\bf 1}_{\{   \nu-\mu <\w < \mu+\nu      \}} (\w+\mu-\nu)^{\delta(\frac 12-q)} (\w+\nu-\mu)^{\delta(\frac 12-q)} 
 \notag \\
 &\qquad= {\bf 1}_{\{   \w < t-s     \}}  \int_{\frac{t+s}{2} }^{\frac{t+s+\w}{2}}  dr \, (\w - t - s + 2r)^{\delta(\frac 12-q)} (\w+t+s-2r)^{\delta(\frac 12-q)} \notag  \\
 &\qquad=  {\bf 1}_{\{   \w < t-s     \}} 2^{\delta(1-2q)} \int_b^c (c-r)^{\delta(\frac 12-q)} (r-a)^{\delta(\frac 12-q)}dr\\  \notag
 &\qquad  \leq  {\bf 1}_{\{   \w < t-s     \}} 2^{\delta(1-2q)} (c-a)^{\delta(1-2q)+1}  \text{Beta}(\delta(\frac 12-q)+1, \delta(\frac 12-q)+1) \\ \notag
&\qquad \lesssim  {\bf 1}_{\{   \w < t-s     \}} \w^{\delta(1-2q)+1} \leq (t-s)^{\delta(1-2q)+1}{\bf 1}_{\{ \|z\| <t-s  \}},
   \end{align}
where  $a= \dfrac{t+s-\w}{2}<b=\dfrac{t+s}{2} < c=\dfrac{t+s+\w}{2}$.
 Thus,   we obtain 
 \begin{equation} \label{K3}
    K^{(3)}_{s,t}(z) \lesssim   (t-s)^{\delta(1-2q)+1}  \big[ (t-s)^2 - \|z\|^2\big]^{\delta(\frac 12-q)}_+ {\bf 1}_{\{ \|z\| <t-s  \}},
    \end{equation}
     with $\delta(q-\frac 12) \le 1-\frac 1{2q} \in(0,\frac12)$. Combining the  estimates \eqref{K1}, \eqref{K2} and \eqref{K3} allows us to finish the proof. 
 \qedhere
  \end{proof}

\end{document}